\documentclass[11pt]{amsart}

\pdfoutput=1

\usepackage{amsmath} 
\usepackage[foot]{amsaddr}
\usepackage{tikz,xcolor}

\usepackage{cancel}
\usepackage{esint,amssymb} 
\usepackage{graphicx}
\usepackage{MnSymbol}
\usepackage{mathtools} 
\usepackage[colorlinks=true, pdfstartview=FitV, linkcolor=blue, citecolor=blue, urlcolor=blue,pagebackref=false]{hyperref}
\usepackage{orcidlink}
\usepackage{microtype}

\usepackage{bm}
\usepackage{dsfont}
\usepackage{mathrsfs} 

\definecolor{darkgreen}{rgb}{0,0.5,0}
\definecolor{darkblue}{rgb}{0,0,0.7}
\definecolor{darkred}{rgb}{0.9,0.1,0.1}

\definecolor{lime}{HTML}{A6CE39}
\DeclareRobustCommand{\orcidicon}{%
	\begin{tikzpicture}
	\draw[lime, fill=lime] (0,0) 
	circle [radius=0.16] 
	node[white] {{\fontfamily{qag}\selectfont \tiny ID}};
	\draw[white, fill=white] (-0.0625,0.095) 
	circle [radius=0.007];
	\end{tikzpicture}
	\hspace{-2mm}
}

\foreach \x in {A, ..., Z}{%
	\expandafter\xdef\csname orcid\x\endcsname{\noexpand\href{https://orcid.org/\csname orcidauthor\x\endcsname}{\noexpand\orcidicon}}
}

\newtheorem{proposition}{Proposition}
\newtheorem{theorem}[proposition]{Theorem}
\newtheorem{lemma}[proposition]{Lemma}

\theoremstyle{remark}
\newtheorem{remark}[proposition]{Remark}

\theoremstyle{definition}
\newtheorem{definition}[proposition]{Definition}

\numberwithin{equation}{section}
\numberwithin{proposition}{section}
\numberwithin{figure}{section}
\numberwithin{table}{section}

\newcommand{\N}{\mathbb{N}}
\newcommand{\Q}{\mathbb{Q}}
\newcommand{\R}{\mathbb{R}}

\newcommand{\E}{\mathbb{E}}

\renewcommand{\leq}{\leqslant}
\renewcommand{\geq}{\geqslant}

\renewcommand{\subset}{\subseteq}
\renewcommand{\bar}{\overline}
\renewcommand{\tilde}{\widetilde}

\renewcommand{\d}{\mathrm{d}}

\newcommand{\mcl}{\mathcal}

\newcommand{\mfk}{\mathfrak}
\newcommand{\msc}{\mathscr}

\newcommand{\upa}{\uparrow}

\begin{document}

\author{Victor Issa\,\orcidlink{0009-0009-1304-046X}}
\address[Victor Issa]{Department of Mathematics, ENS de Lyon, Lyon, France}
\email{\href{mailto:victor.issa@ens-lyon.fr}{victor.issa@ens-lyon.fr}}

\subjclass[2020]{}

\title[On Permutation-Invariant Optimizers for Parisi Formula]{Existence and Uniqueness of Permutation-Invariant Optimizers for Parisi Formula}

\begin{abstract}
    It has recently been shown in \cite{bates2023parisi} that, upon constraining the system to stay in a balanced state, the Parisi formula for the mean-field Potts model can be written as an optimization problem over permutation-invariant functional order parameters. 
    
    In this paper, we focus on permutation-invariant mean-field spin glass models. After introducing a correction term in the definition of the free energy and without constraining the system, we show that the limit free energy can be written as an optimization problem over permutation-invariant functional order parameters. We also show that for some models this optimization problem admits a unique optimizer. In the case of Ising spins, the correction term can be easily removed, and those results transfer to the uncorrected limit free energy.   
    
    We also derive an upper bound for the limit free energy of some nonconvex permutation-invariant models. This upper bound is expressed as a variational formula and is related to the solution of some Hamilton-Jacobi equation. We show that if no first order phase transition occurs, then this upper bound is equal to the lower bound derived in \cite{mourrat2020free}. We expect that this hypothesis holds at least in the high temperature regime.
    
    Our method relies on the fact that the free energy of any convex mean-field spin glass model can be interpreted as the strong solution of some Hamilton-Jacobi equation.

    \bigskip


    \medskip

\end{abstract}

\maketitle

\newpage
\thispagestyle{empty}
{
  \hypersetup{linkcolor=black}
  \tableofcontents
}

\newpage
 \pagenumbering{arabic}
\section{Introduction}

\subsection{Preamble}

Let $D > 1$ be an integer, that we will keep fixed throughout the paper. We study Gaussian processes $(H_N(\sigma))_{\sigma \in \R^{D \times N}}$, whose covariance is of the form
\begin{equation} \label{e. covariance}
    \E [H_N(\sigma) H_N(\tau)] = N\xi \left( \frac{\sigma \tau^*}{N}\right).
\end{equation}
Here $\sigma \tau^* =(\sigma_d \cdot \tau_{d'})_{1 \leq d,d' \leq D}$, $x \cdot y$ is the standard scalar product on $\R^N$ and $\xi : \R^{D \times D} \to \R$ is a function given by an absolutely convergent power series. Unless stated otherwise, we will always assume that $\xi$ is convex on the set of positive semi-definite matrices. We give ourselves for each $N$ a reference probability measure $P_N$ on $\R^{D \times N}$ of the form $P_N = P_1^{\otimes N}$ where $P_1$ is a compactly supported probability measure on $\R^D$. For every $t \geq 0$, One can associate a random probability measure on $\R^{D \times N}$ to the process $H_N$ called the Gibbs measure and denoted by $\langle \cdot \rangle$. It is defined by
\begin{equation*}
    \langle h(\sigma) \rangle = \frac{\int h(\sigma) \exp \left( \sqrt{2t}H_N(\sigma) - Nt \xi \left( \frac{\sigma \sigma^*}{N}\right) \right) \d P_N(\sigma) }{\int \exp \left( \sqrt{2t}H_N(\sigma) - Nt \xi \left( \frac{\sigma \sigma^*}{N}\right) \right) \d P_N(\sigma)}.
\end{equation*}
An important step, in understanding the Gibbs measure associated to the family of processes $(H_N)_N$ is the computation of the large $N$ limit of the free energy 
\begin{equation} \label{e.free energy}
    \bar F_N(t) = - \frac{1}{N} \E \log \int \exp \left( \sqrt{2t}H_N(\sigma) - Nt \xi \left( \frac{\sigma \sigma^*}{N}\right) \right) \d P_N(\sigma).
\end{equation}
For the models of interest here, a variational formula for the limiting value of $\bar F_N(t)$ is known, this is the celebrated Parisi formula. 
The Parisi formula was first conjectured in \cite{parisi1979infinite} using a sophisticated non-rigorous argument now referred to as the replica method. The convergence of the free energy as $N \to +\infty$ was rigorously established in \cite{guerra2002} in the case of the so-called Sherrington-Kirkpatrick model which corresponds to $D = 1$, $\xi(x) = x^2$ and $P_1 = \text{Unif}(\{-1,1\})$. The Parisi formula for the Sherrington-Kirkpatrick model was then proven in \cite{gue03,Tpaper}. This was extended to the case $D = 1$, $P_1 = \text{Unif}(\{-1,1\})$ and $\xi(x) = \sum_{p \geq 1} a_p x^p$ with $a_p \geq 0$ in \cite{pan}. Some models with $D > 1$ such as multispecies models, the Potts model, and a general class of models with vector spins were treated in \cite{pan.multi,pan.potts,pan.vec}, under the assumption that $\xi$ is convex on $\R^{D \times D}$. Finally, the case $D > 1$ was treated in general in \cite{chenmourrat2023cavity} assuming only that that $\xi$ is convex on the set of positive semi-definite matrices. The following version of the Parisi formula is \cite[Corollary~8.2]{chenmourrat2023cavity}.

\begin{theorem} [\cite{chenmourrat2023cavity}] \label{t.parisi}
    If $\xi$ is convex on $S^D_+$, then for every $t > 0$,
    \begin{equation} \label{e.parisi}
        \lim_{N \to +\infty}  \bar F_N(t) = \sup_{\mfk q}\left\{ \psi(\mfk q) - t \int_0^1 \xi^* \left( \frac{ \mfk q(u)}{t}\right) \d u \right\}.
    \end{equation}
\end{theorem}

Here $S^D_+$ denotes the set of positive semi-definite symmetric matrices in $\R^{D \times D}$. We equip $\R^{D \times D}$ with the order $A \leq B$ if and only $B-A \in S^D_+$. The supremum in \eqref{e.parisi} is taken over the set of bounded functions in
\begin{equation}
    \mcl Q(S^D_+) = \{ \mfk q : [0,1) \to S^D_+ \big| \; \text{$\mfk q$ is càdlàg and nondecreasing} \}.
\end{equation}
The function $\psi$ is the cascade transform of the measure $P_1$ and $\xi^*$ is the convex conjugate of $\xi$ with respect to the cone $S^D_+$. We postpone the precise definition of these objects to \eqref{e.def psi} and \eqref{e.convex conjugate} respectively.

Let $\msc S_D$ denote the group of permutations of $\{1,\dots,D\}$, we say that $\xi$ is permutation-invariant and $P_1$ is permutation-invariant when for every permutation $s \in \msc S_D$,
\begin{equation} \label{e. xi permutation invariance}
    \forall R \in \R^{D \times D},  \; \xi(R) = \xi \left( (R_{s(d),s(d')})_{1 \leq d,d' \leq D}\right),
\end{equation}
\begin{equation} \label{e. P1 permutation invariance}
    \forall \chi \in \mcl C_b(\R^{D}), \; \int \chi(x_1,\dots,x_D) \d P_1(x) = \int \chi(x_{s(1)},\dots,x_{s(D)}) \d P_1(x).
\end{equation}
We will show that when $\xi$ and $P_1$ are permutation-invariant, the supremum in \eqref{e.parisi} can be taken over the set of permutation-invariant paths in $\mcl Q(S^D_+)$. This result can be interpreted as an absence of breaking of permutation-invariance by the system. This persistence of permutation-invariance was predicted to happen for the Potts model (see \eqref{e. Potts model} below) in \cite{elderfield1983potts}. It has been rigorously proven that the Potts model does not break its permutation invariance, but only after constraining the system to stay in a balanced state \cite{bates2023parisi} or by introducing a correction term of the form $-Nt \xi \left( \frac{\sigma \sigma^*}{N}\right)$ \cite{chen2023parisi}. In \eqref{e.free energy} the correction term of \cite{chen2023parisi} is also present, but our result is general enough to cover models where $P_1$ is supported on $\{-1,1\}^D$ and $\xi$ only depends on the diagonal coefficients of its argument, for those models the correction term is constant and our main results transfer to the uncorrected free energy. 
\subsection{Main results}
Let $\mcl Q(\R_+)$ denote the set of càdlàg and nondecreasing functions $[0,1) \to \R_+$. Let $U$ denote a uniform random variable in $[0,1)$, given $p_1,p_2 \in \mcl Q(\R_+)$, we let $(p_1,p_2)^\perp : [0,1) \to S^D_+$ be defined by 
\begin{equation*}
    (p_1,p_2)^\perp = p_1 \left( \text{id}_D - \frac{\mathbf{1}_D}{D} \right) +  p_2\frac{\mathbf{1}_D}{D},
\end{equation*}
 where $\text{id}_D$ denotes the $D \times D$ identity matrix and $\mathbf{1}_D$ the $D \times D$ matrix whose coefficients are all equal to $1$. As will be proven in Section~\ref{s.generalities}, every permutation-invariant path in $\mcl Q(S^D_+)$ is of the form $(p_1,p_2)^\perp$. 
 
 Heuristically, given a maximizing path $\mfk q$ in \eqref{e.parisi}, the law of the matrix $\mfk q(U)$ is the limiting distribution of $\frac{\sigma \tau^*}{N}$ where $\sigma$ and $\tau$ are two independent random variables with law $\langle \cdot \rangle$. For some specific models, the distribution of the overlap matrix has some additional properties, and those additional properties allow us to write the limit free energy as an optimization over a smaller set of paths. For example, consider the Potts model which corresponds to the family of processes
\begin{equation} \label{e. Potts model}
    H^\text{Potts}_N(\sigma) = \frac{1}{\sqrt{N}} \sum_{i,j=1}^N J_{ij} \sigma_i \cdot \sigma_j.
\end{equation}
Here $(J_{ij})_{i,j \geq 1}$ denotes a family of independent standard Gaussian random variables. The process $H_N^\text{Potts}$ satisfies \eqref{e. covariance} with $\xi(R) = \sum_{d,d'=1}^D R^2_{dd'}$. Often, the Potts model is considered with reference measure $P_1 = \text{Unif} \{e_1,\dots,e_D\}$, where $(e_1,\dots,e_D)$ denote the canonical basis of $\R^D$. With this assumption, the terms of the form $\sigma_i \cdot \sigma_j$ appearing in \eqref{e. Potts model} only take the value $0$ or $1$. In this case, if we sample two independent random variables $\sigma, \tau \in \R^{D \times N}$ with law $\langle \cdot \rangle$, their overlap matrix $R = \frac{\sigma \tau^*}{N}$ satisfies $\sum_{d,d'=1}^D R_{dd'} = 1$ almost surely and for every $s \in \msc S_D$, the matrices $R$ and $(R_{s(d)s(d')})_{1 \leq d,d' \leq D}$ are equal in law under $\E \langle \cdot \rangle$. This means that if $\mfk q$ is a maximizing path for the Parisi formula of the Potts model, we should expect that 
\begin{equation*}
    \mfk q = p \left( \text{id}_D - \frac{1_D}{D} \right) + \frac{\mathbf{1}_D}{D^2} = \left(p, \frac{1}{D} \right)^\perp
\end{equation*}
for some $p \in \mcl Q(\R_+)$. This observation on the set of optimal paths in \eqref{e.parisi} for the Potts model was leveraged in \cite{chen2023parisi, bates2023parisi}. In \cite{bates2023parisi}, the authors show that when the system is constrained to stay in a balanced state, the limit free energy of the Potts model can be written a supremum over $\mcl Q(\R_+)$. In \cite{chen2023parisi}, the author does not constrain the system but introduces a correction term of the form $-Nt \xi \left(\frac{\sigma \sigma^*}{N} \right)$ like in \eqref{e.free energy} and obtains results similar to \cite{bates2023parisi}.

In this paper, we show that similar results can be obtained in different settings. We will focus on permutation-invariant models, that is models with $\xi$ and $P_1$ satisfying \eqref{e. xi permutation invariance} and $\eqref{e. P1 permutation invariance}$. We will show that for those models, the following variational formula holds.

\begin{theorem} \label{t.symmetric optimizer matrix}
    Assume that $\xi$ is convex on $S^D_+$ and permutation-invariant, assume that $P_1$ is permutation-invariant. Then, for every $t \geq 0$,  
    \begin{equation} \label{e.symmetric optimizer matrix }
        \begin{split}
           \lim_{N \to +\infty} \bar F_N(t) = \sup_{(p_1,p_2)} &\inf_{(r_1,r_2)}\Biggl\{ \psi( (p_1,p_2)^\perp ) \\ -\langle p_1,r_1\rangle_{L^2}
                                                                                       &- \langle p_2,r_2 \rangle_{L^2} + t \int_0^1 \xi\left( \left(\frac{p_1(u)}{D-1},p_2(u)\right)^\perp \right) \d u\Biggl\}. 
        \end{split}        
    \end{equation}
    Where the supremum and the infimum are taken over $(\mcl Q(\R_+) \cap L^\infty)^2$.
\end{theorem}

Similarly to the Potts model, at the heuristic level \eqref{e.symmetric optimizer matrix } follows from the fact that given two independent random variables $\sigma$ and $\tau$ under $\langle \cdot \rangle$, for every $s \in \msc S_D$, the matrices $\frac{\sigma \tau^*}{N}$ and $\left(\frac{\sigma_{s(d)} \cdot \tau_{s(d')}}{N}\right)_{1 \leq d,d' \leq D}$ are equal in law under $\E \langle \cdot \rangle$. This means that every maximizing path should be of the form $\mfk q =(p_1,p_2)^\perp$ where $p_1,p_2 \in \mcl Q(\R_+)$ and the limit free energy should be a supremum over $\mcl Q(\R_+)^2$. 

We will also consider permutation-invariant models, where the interaction function $\xi$ is further assumed to only depend on $\frac{\sigma_1 \cdot \tau_1}{N},\dots,\frac{\sigma_D \cdot \tau_D}{N}$. In this case, we identify the map $R \mapsto \xi(R)$ defined on $\R^{D \times D}$ with the map $x \mapsto \xi(\text{diag}(x))$ defined on $\R^D$. Since $\xi$ only depends on the diagonal coefficients of its argument, we only need to keep track of the diagonal of the overlap matrix and the limit free energy can be written as a supremum over $\mcl Q(\R_+)^D$. Performing the same heuristic computation as above, we then expect that under our permutation invariance assumptions, the limit free energy can in fact be written as a supremum over $\mcl Q(\R_+)$. We will show that this is indeed the case.

Given a maximizing path $r$ of the Parisi formula, the law of the random variable $r(U)$ is called a Parisi measure. When $D = 1$, it is known that there exists a unique Parisi measure \cite{auffinger2015parisi}. In this case, the proof relies on a strict concavity property of the Parisi functional. However, when $D > 1$, for technical reasons this strict concavity property does not carry over well to $\mcl Q(S^D_+)$ or $\mcl Q(\R_+)^D$. But, in \eqref{e.symmetric optimizer vector} below, the Parisi formula is written as a maximization over $\mcl Q(\R_+)$, the set of $1$-dimensional paths. Thanks to this, we can proceed as in \cite{auffinger2015parisi} to prove uniqueness of Parisi measures. 
\begin{theorem} \label{t.symmetric optimizer vector}
    Assume that $\xi$ is convex on $S^D_+$, permutation-invariant and only depends on the diagonal coefficients of its argument, assume that $P_1$ is permutation-invariant. Then, for every $t > 0$,
    \begin{equation} \label{e.symmetric optimizer vector}
        \lim_{N \to +\infty} \bar F_N(t) = \sup_{p \in \mcl Q(\R_+) \cap L^\infty} \left\{ \psi(p \text{id}_D) - t \int_0^1 \xi^* \left( \frac{ p(u) \text{id}_D}{t} \right) \d u \right\}.
    \end{equation}
    In addition, the supremum in \eqref{e.symmetric optimizer vector} is reached at a unique $p^* \in \mcl Q(\R_+) \cap L^\infty$.  
\end{theorem}
Let us now comment on the inclusion of the term $-Nt \xi \left( \frac{\sigma \sigma^*}{N}\right)$ in the definition of the free energy \eqref{e.free energy}. Consider again the Potts model, that is $\xi(R) = \sum_{d,d' = 1}^D R^2_{dd'}$ and $P_1 = \text{Unif}\{e_1,\dots,e_D\}$. For any $\sigma \in \{e_1,\dots,e_D\}^N$, we have $\frac{\sigma \sigma^*}{N} = \text{diag}(\alpha_1,\dots,\alpha_D)$ where
\begin{equation*}
    \alpha_d = \frac{1}{N}\#\{ i \leq N, \sigma_i = e_d\}.
\end{equation*}
So, $\xi(\sigma \sigma^*/N) = \sum_{d= 1}^D \alpha^2_d$ with $\sum_{d =1}^D \alpha_d = 1$ and $\alpha_d \geq 0$. This means that for the Potts model, the value of the correction term $-Nt \xi \left( \frac{\sigma \sigma^*}{N}\right)$ is minimal on configurations of the form $\sigma = (e_d,\dots,e_d)$ and maximal on configurations satisfying $\alpha_d = \frac{1}{D}$. In particular, at least in the case of the Potts model, even though we are not constraining the system to stay in a balanced state, the correction term favors configurations $\sigma \in \R^{D \times N}$ which are balanced. Thankfully, this is not always the case and for some models the correction term $-Nt \xi \left( \frac{\sigma \sigma^*}{N}\right)$ can be removed. For example, if we assume that $\xi$ only depends on the diagonal coefficients of its argument and $P_1$ is supported on $\{-1,1\}^D$, then the correction term is constant. In this case, the variational formula \eqref{e.symmetric optimizer vector} can be rewritten as a variational formula for the uncorrected free energy.

Let $\alpha \geq 1$, for $\sigma \in \R^{2 \times N}$ define 
\begin{equation*}
    H_N^\text{BP+SK}(\sigma) = \frac{1}{\sqrt{N}} \sum_{i,j =1}^N J^{11}_{ij}\sigma_{1i}\sigma_{1j} + \frac{1}{\sqrt{N}} \sum_{i,j =1}^N J^{22}_{ij}\sigma_{2i}\sigma_{2j} + \frac{1}{\sqrt{N}} \sum_{i,j =1}^N J^{12}_{ij}\sigma_{1i}\sigma_{2j},
\end{equation*}
where $J^{11} = (J^{11}_{ij})_{i,j \geq 1}$ and $J^{22} = (J^{22}_{ij})_{i,j \geq 1}$ are independent families of independent centered Gaussian random variables with variance $\alpha/2$ and $(J^{12}_{ij})_{i,j \geq 1}$ is a family of independent Gaussian random variables with variance $1$ and independent of $J^{11}$ and $J^{22}$. The process $ H_N^\text{BP+SK}$ satisfies \eqref{e. covariance} with,
\begin{equation*}
     \xi^{BP+SK}(R) = \frac{\alpha}{2} R_{11}^2 + \frac{\alpha}{2} R_{22}^2 + R_{11}R_{22}.
\end{equation*}
Choosing $P_1 = \text{Unif}(\{-1,1\}^2)$ as the reference probability measure, Theorem~\ref{t.symmetric optimizer vector} can be applied to $H_N^\text{BP+SK}$ to discover that the system does not break its permutation invariance. One can wonder what happens when $0 \leq \alpha < 1$, in this case $\xi^{BP+SK}$ is nonconvex on $S^2_+$ and there is no known generalization of Theorem~\ref{t.parisi}. In Section~\ref{s.upper bound}, we will discuss some results that are applicable in this case.

When $\xi$ is not assumed to be convex on $S^D_+$, to the best of our knowledge, there is no proof of the fact that $\bar F_N(t)$ converges as $N \to +\infty$. Building upon an interpolation argument developed in \cite[Section~2.1]{barra2012glassy}, we will show that for permutation-invariant models where $\xi$ only depends on the diagonal coefficients of its argument and is nonconvex on $S^D_+$, the right-hand side in \eqref{e.symmetric optimizer vector} is an upper bound on the $\limsup$ of the free energy.

\begin{theorem} \label{t. upper bound}
    Assume that $\xi$ is permutation-invariant and only depends on the diagonal coefficients of its argument, assume that $P_1$ is permutation-invariant. Then, even when $\xi$ is nonconvex, we have for every $t > 0$,
    \begin{equation*}
        \limsup_{N \to +\infty} \bar F_N(t) \leq \sup_{p \in \mcl Q(\R_+) \cap L^\infty} \left\{ \psi(p \text{id}_D) - t \int_0^1 \Xi^* \left( \frac{ p(u) \text{id}_D}{t} \right) \d u \right\},
    \end{equation*}
    where $\Xi(x) = \frac{1}{D} \sum_{d = 1}^D \xi(x_d,\dots,x_d)$.
\end{theorem}

In addition, we show that if no first order phase transition occurs, then the upper bound of Theorem~\ref{t. upper bound} matches the lower bound derived in \cite{mourrat2020free,mourrat2020nonconvex}. The no first order phase transition hypothesis is formally written in \eqref{e.hypo} below, note that this hypothesis is true when $\xi$ is convex. We refer to Theorem~\ref{t.nonconvex fe} below for a formal version of Theorem~\ref{t.informal nonconvex fe}. 
\begin{theorem} \label{t.informal nonconvex fe} 
   Assume that the hypotheses of Theorem~\ref{t. upper bound} hold and further assume that no first order phases transition occurs, then for every $t > 0$, $\bar F_N(t)$ converges as $N \to +\infty$ and 
    \begin{equation*}
        \lim_{N \to +\infty} \bar F_N(t) = \sup_{p \in \mcl Q(\R_+) \cap L^\infty} \left\{ \psi(p \text{id}_D) - t \int_0^1 \Xi^* \left( \frac{ p(u) \text{id}_D}{t} \right) \d u \right\},
    \end{equation*}
    where $\Xi(x) = \frac{1}{D} \sum_{d = 1}^D \xi(x_d,\dots,x_d)$.
\end{theorem}
As explained in Remark~\ref{r.no first order transition for small t} below, it is plausible that there exists $t_c > 0$ such that no first order phase transition occurs on $[0,t_c)$. In this case, the argument used to prove Theorem~\ref{t.informal nonconvex fe} allows to identify the limit free energy on $[0,t_c]$. 
\subsection{Motivations}
Most of the proof we derive below rely on the following result, which we state informally. 
\begin{theorem}[\cite{mourrat2019parisi,chenmourrat2023cavity}] \label{t.visco and fe}
    Assume that $\xi$ is convex on $S^D_+$, then for every $t \geq 0$,
    \begin{equation*}
        \lim_{N \to +\infty} \bar F_N(t) = f(t,0),
    \end{equation*}
    where $f : \R_+ \times \mcl Q(S^D_+) \to \R$ is the unique solution of 
    \begin{equation} \label{e. hj organization}
        \begin{cases}
            \partial_t f  - \int \xi(\nabla f) = 0 \text{ on } \R_+ \times\mcl Q(S^D_+)\\
            f(0,\cdot) = \psi \text{ on } \mcl Q(S^D_+).
        \end{cases}
    \end{equation}
\end{theorem}

In Section~\ref{s.enriched } we will give a precise definition of $f$ and $\psi$, we will also explain the precise meaning of \eqref{e. hj organization}. In the context of Theorem~\ref{t.symmetric optimizer matrix} and Theorem~\ref{t.symmetric optimizer vector}, this approach trough partial differential equations does not seem to possess any clear advantage when compared with the methods developed in \cite{chen2023parisi}. The true power of Theorem~\ref{t.visco and fe} is revealed when dealing with nonconvex models. When $\xi$ is nonconvex, the Parisi formula completely breaks down. To the best of our knowledge, until recently, it seems that there was no clear conjecture on what the limit of the free energy should be in this case. In \cite[Conjecture~2.6]{mourrat2019parisi}, it is proposed that results such as Theorem~\ref{t.visco and fe} should generalize to nonconvex models. Further development in \cite{mourrat2020free,mourrat2020nonconvex} have led to the following lower bound for the free energy of nonconvex models.

\begin{theorem}[\cite{mourrat2020free,mourrat2020nonconvex}] \label{t.lower bound}
    For every $t \geq 0$,
    \begin{equation*}
        \liminf_{N \to +\infty} \bar F_N(t) \geq f(t,0),
    \end{equation*}
    where $f : \R_+ \times \mcl Q(S^D_+) \to \R$ is the unique solution of 
    \begin{equation} 
        \begin{cases}
            \partial_t f  - \int \xi(\nabla f) = 0 \text{ on } \R_+ \times\mcl Q(S^D_+)\\
            f(0,\cdot) = \psi \text{ on } \mcl Q(S^D_+).
        \end{cases}
    \end{equation}
\end{theorem}

To identify the limit free energy of nonconvex models, a possible approach is to prove an upper bound for $\limsup_{N \to +\infty} \bar F_N(t)$ and compare it with $f(t,0)$. In Section~\ref{s.upper bound}, we will show that 
\begin{equation*}
    f(t,0) \leq \liminf_{N \to +\infty} \bar F_N(t)  \leq \limsup_{N \to +\infty} \bar F_N(t) \leq g(t,0),
\end{equation*}
where $g$ is the solution of an equation similar to \eqref{e. hj organization} (see \eqref{e.HJ nonconvex vector} below). We are able to phrase the upper bound this way thanks to the fact that the proofs of Theorem~\ref{t.symmetric optimizer matrix} and Theorem~\ref{t.symmetric optimizer vector} are formalized in the language of partial differential equations. In particular, we will show that under some unproven regularity assumption on $f$, $f(t,0)$ and $g(t,0)$ are equal. We refer to Section~\ref{ss. nonconvex hj} for a more precise discussion.
\subsection{Organization of the paper}
We only consider permutation-invariant models. That is, models with $\xi$ permutation-invariant and $P_N = P_1^{\otimes N}$ with $P_1$ permutation-invariant and compactly supported. Without loss of generality, we will assume that $P_1$ is supported inside the unit ball of $\R^D$. Excluding Section~\ref{s.upper bound}, $\xi$ is always assumed to be convex on $S^D_+$. 

In Section~\ref{s.enriched } we will give a precise definition of the functions $f$ and $\psi$ appearing in Theorem~\ref{t.visco and fe}, we will also explain the precise meaning of \eqref{e. hj organization}. In Section~\ref{s.hj} we will define an appropriate notion of solution for \eqref{e. hj organization} following \cite{chen2023viscosity}. In particular, in Section~\ref{s.enriched } we introduce a variant of the free energy $\bar F_N(t)$ which will depend on $t \in \R_+$ and an extra parameter $\mfk q \in \mcl Q(S^D_+)$. This enriched version of the free energy will be denoted $\bar F_N(t,\mfk q)$ and satisfies $\bar F_N(t,0) = \bar F_N(t)$. It is through the introduction of this extra parameter that it is possible to obtain a partial differential equation for the limiting objects. 
In section~\ref{s.generalities} we will introduce some basic notations on permutation-invariant matrices and paths that will be useful for the later sections. We make the elementary but important remark that the set of permutation-invariant matrices is isomorphic to $\R^2$ and the set of permutation-invariant vectors is isomorphic to $\R$. The bulk of our analysis is in Section~\ref{s. fe and hj}, in this section we start by showing that the enriched free energy is permutation-invariant, namely for every permutation $s \in \msc S_D$,
\begin{equation*}
    \bar F_N(t, (\mfk q_{dd'})_{1 \leq d,d \leq D}) = \bar F_N(t,(\mfk q_{s(d)s(d')})_{1 \leq d,d \leq D}).
\end{equation*}
Thanks to this property and \eqref{e. hj organization}, we can write a partial differential equation satisfied by $(t,\mfk q) \mapsto \lim_{N \to +\infty} \bar F_N(t,\mfk q)$ on the set of permutation-invariant paths. Since the set of permutation-invariant paths in $\mcl Q(S^D_+)$ is isomorphic to $\mcl Q(\R^2_+)$, we will deduce Theorem~\ref{t.symmetric optimizer matrix}. Moreover, when $\xi$ is further assumed to only depend on the diagonal coefficient of its argument, it is even possible to write a partial differential equation on the set of permutation-invariant paths in $\mcl Q(\R_+)^D$ which is isomorphic to $\mcl Q(\R_+)$, this will yield the variational formula \eqref{e.symmetric optimizer vector} of Theorem~\ref{t.symmetric optimizer vector}. Using a strict concavity property of the Parisi functional, in Section~\ref{s.uniqueness} we will deduce the existence and the uniqueness of an optimizer in the variational formula \eqref{e.symmetric optimizer vector}. Finally, in Section~\ref{s.upper bound} we assume that $\xi$ is permutation-invariant and depends only on the diagonal coefficients of its argument, but we do not assume that $\xi$ is convex. We will show that under those hypotheses if $\xi$ admits an absolutely convergent power series and  satisfies \eqref{e. covariance}, then 
\begin{equation*}
    \forall x \in \R^D_+, \; \xi(x) \leq \frac{1}{D} \sum_{d=1}^D \xi(x_d,\dots,x_d).
\end{equation*}
With this inequality, we will prove Theorem~\ref{t. upper bound} via a simple interpolation argument. To conclude, we will explore how this upper bound is related to the lower bound for the free energy of nonconvex models derived in \cite{mourrat2020free,mourrat2020nonconvex} and we will prove a rigorous version of Theorem~\ref{t.informal nonconvex fe}, namely Theorem~\ref{t.nonconvex fe}.
\subsection{Acknowledgement}
I warmly thank Jean-Christophe Mourrat for countless fascinating discussions on spin glasses and many helpful comments and insights during this project.
\section{The enriched free energy} \label{s.enriched }
The method we propose relies on the fact that the limit free energy $f(t) = \lim \bar F_N(t)$ solves a partial differential equation. At this stage, $f$ only depends on one parameter, in order to obtain a partial differential equation, we are going to introduce an extra parameter. The nature of this parameter will not be the same for models that satisfy \eqref{e. covariance} and models that satisfy \eqref{e. covariance} under the additional assumption that $\xi$ only depends on the diagonal coefficients of its argument. We will provide the construction for the first class of models and then explain the additional adjustments that can be made when $\xi$ only depends on the diagonal coefficients of its argument.

Here $S^D_+$ denotes the set of positive semi-definite symmetric matrices in $\R^{D \times D}$ and $S^D_{++}$ the set of positive symmetric matrices in $\R^{D \times D}$. We equip $R^{D \times D}$ with the order $A \leq B$ if and only $B-A \in S^D_+$. The extra parameter we introduce in the definition of the free energy belongs to the space
\begin{equation}
    \mcl Q(S^D_+) = \{ \mfk q : [0,1) \to S^D_+ \big| \; \text{$\mfk q$ is càdlàg and nondecreasing} \}.
\end{equation}
In general paths in $\mcl Q(S^D_+)$ are not defined at $u = 1$, but if $\mfk q \in \mcl Q(S^D_+) \cap L^\infty$ we can set
\begin{equation*}
    \mfk q(1) = \lim_{u \upa 1} \mfk q(u) \in S^D_+.
\end{equation*}
We let $\mfk R$ be a Poisson Dirichlet cascade, briefly $\mfk R$ is a specific random probability measure on the unit sphere of a separable Hilbert space $(\mfk H, \wedge)$, such that given two independent random variables $\alpha,\alpha' \in \mfk H$ sampled from $\mfk R$, the random variable $\alpha \wedge \alpha'$ is uniformly distributed in $[0,1]$ (see \cite[Section~2]{pan} and \cite[Section~4]{chenmourrat2023cavity}). We let $\langle \cdot \rangle_{\mfk R}$ denote the expectation with respect to $\mfk R^{\otimes \mathbb{N}}$ and $\mfk U$ the support of $\mfk R$. By construction of $\mfk R$, almost surely $\mfk U$ is an ultrametric set, that is almost surely for every $\alpha^1,\alpha^2,\alpha^3 \in \mfk U$,
\begin{equation*}
    \alpha^1 \wedge \alpha^3 \geq \min \left( \alpha^1 \wedge \alpha^2, \alpha^2 \wedge \alpha^3\right).
\end{equation*}

Let $\mfk q \in \mcl Q(S^D_+) \cap L^\infty$, according to \cite[Propsoition~4.1]{chenmourrat2023cavity} $\mfk R$-almost surely, there exists a $\R^D$-valued centered Gaussian process $(w^{\mfk q}(\alpha))_{\alpha \in \mfk U}$ such that
    \begin{equation*}
        \E \left[  w^{\mfk q}(\alpha) \left( w^{\mfk q}(\alpha')\right)^*\right]  = \mfk q(\alpha \wedge \alpha').
    \end{equation*}
    Let $N \geq 1$, and consider $(W_N^{\mfk q}(\alpha))_{\alpha \in \mfk U}$ be a $(\R^{D})^N$-valued random process whose coordinates are independent and have the same law as $(w^{\mfk q}(\alpha))_{\alpha \in \mfk U}$. We set
    \begin{equation} \label{e. enriched fe}
        F_N(t,\mfk q) = -\frac{1}{N} \log \iint \exp(H_N^{t,\mfk q}(\sigma,\alpha)) dP_N(\sigma) d \mfk R(\alpha),
    \end{equation}
    where
    \begin{equation*}
        H_N^{t,\mfk q}(\sigma,\alpha) = \sqrt{2t} H_N(\sigma) -t N \xi \left(\frac{\sigma \sigma^* }{N}\right) + \sqrt{2} W_N^\mfk q(\alpha) \cdot \sigma - \mfk q(1) \cdot \sigma \sigma^*.
    \end{equation*}
   We let $\bar F_N(t,\mfk q) = \E F_N(t,\mfk q)$, the function $\bar F_N : \R_+ \times \left( \mcl Q(S^D_+) \cap L^\infty \right) \to \R$ is Lipschitz, more precisely according to \cite[Proposition~5.1]{chenmourrat2023cavity}, we have 
   \begin{equation} \label{e. lipschitz free energy}
       |\bar F_N(t,\mfk q) - \bar F_N(t',\mfk q')| \leq |\mfk q- \mfk q'|_{L^1} +|t-t'| \sup_{|a| \leq 1} |\xi(a)|.
   \end{equation}
   Here, $|a|$ denotes the Frobenius norm of $a \in \R^{D \times D}$, that is the norm associated to the scalar product
   \begin{equation*}
       a \cdot b = \sum_{d,d'=1}^D a_{dd'}b_{dd'}.
   \end{equation*}
   Using \eqref{e. lipschitz free energy}, we can then extend $\bar F_N$ by continuity to $\R_+ \times \left( \mcl Q(S^D_+) \cap L^1 \right)$. Note that if $0$ denotes the null path, we have $\bar F_N(t,0) = \bar F_N(t)$. So the function we have defined is indeed an extension of the free energy. 
   
   Let us now introduce the notion of differentiability that we will use to formulate the partial differential equation \eqref{e. hj organization}. Given $\mfk q \in \mcl Q(S^D_+) \cap L^2$, we let 
   \begin{equation*}
       \text{Adm}(\mfk q) = \{ \kappa \in L^2 \big| \; \exists r > 0, \; \forall t \in [0,r], \; \mfk q+t \kappa \in \mcl Q(S^D_+) \cap L^2 \}.
   \end{equation*}
   The set $\text{Adm}(\mfk q)$ is the set of admissible directions at $\mfk q$. Let $g : \mcl Q(S^D_+) \cap L^2 \to \R$, we say that $g$ is Gateaux differentiable at $\mfk q \in \mcl Q(S^D_+) \cap L^2$ when for every $\kappa \in \text{Adm}(\mfk q)$, the following limit exists 
   \begin{equation*}
       g'(\mfk q ; \kappa) = \lim_{t \downarrow 0} \frac{g(\mfk q + t \kappa) - g(\mfk q)}{t},
   \end{equation*}
   and there exists a unique $d \in L^2$ such that for every $\kappa \in \text{Adm}(\mfk q)$, 
   \begin{equation*}
        g'(\mfk q ; \kappa) = \langle d, \kappa \rangle_{L^2}.
   \end{equation*}
   In this case, the vector $d \in L^2$ is denoted $\nabla g(\mfk q)$ and we call it the Gateaux derivative of $g$ at $\mfk q$. Note that $\nabla g(\mfk q) \in L^2$, so we can make sense of expressions of the form $\int_0^1 \xi( \nabla g(\mfk q)(u)) \d u$, provided that the integral converges. Such expressions will be simply abbreviated by $\int \xi(\nabla g(\mfk q))$ in what follows.
   
   We define $\mcl Q_\upa(S^D_+)$ as the set of paths $\mfk q \in \mcl Q(S^D_+) \cap L^2$ such that $\mfk q(0) = 0$, and there exists a constant $c > 0$ such that $u \mapsto \mfk q(u) -c u \text{id}_D$ is nondecreasing and for all $u < v$, 
\begin{equation*}
   \text{Ellipt}(\mfk q(v)-\mfk q(u)) \leq \frac{1}{c}.
\end{equation*}
 Here, for $m \in S^D_{++}$, $\text{Ellipt}(m)$ denotes the ratio of the biggest and smallest eigenvalue of $m$. When $\xi$ is convex, which will always be the case except in Section~\ref{s.upper bound}, it can be shown that $\bar F_N$ converges pointwise to some Lipschitz function $f: \R_+ \times \left( \mcl Q(S^D_+) \cap L^1 \right) \to \R$. Furthermore, according to \cite[Propositions~7.2~\&~8.6]{chenmourrat2023cavity}, $f : \R_+ \times (\mcl Q(S^D_+) \cap L^2) \to \R$ is Gateaux differentiable on $(0,+\infty) \times (\mcl Q_\upa(S^D_+) \cap L^\infty)$ and satisfies
\begin{equation} \label{e. HJ smooth}
    \begin{cases}
        \partial_t f - \int \xi(\nabla f) = 0 &\text{ on } (0,+\infty) \times (\mcl Q_\upa(S^D_+) \cap L^\infty) \\
        f(0,\cdot) = \psi  &\text{ on } \mcl Q_\upa(S^D_+) \cap L^\infty,
    \end{cases}
\end{equation}
where we have defined $\psi = \lim_{N \to +\infty} \bar F_N(0,\cdot)$. In fact, since we have assumed that $P_N = P_1^{\otimes N}$ we simply have
\begin{equation} \label{e.def psi}
    \psi = \bar F_1(0,\cdot).
\end{equation}
The fact that $f$ is a solution of \eqref{e. HJ smooth} at every point where it is differentiable is a consequence of the cavity computations conducted  in \cite{chenmourrat2023cavity} (see \cite[Propostion~7.2]{chenmourrat2023cavity}), this result is robust and holds true even when $\xi$ is nonconvex. The real miracle is the fact that $f$ is differentiable at every (nice) points, this is proven in \cite[Propostion~8.6]{chenmourrat2023cavity} and relies strongly on the fact that $\xi$ is convex on $S^D_+$. We stress that in general, nonlinear partial differential equations such as  \eqref{e. HJ smooth} do not always have a differentiable solution. Therefore, to study \eqref{e. HJ smooth} and other similar equations, we will need to appeal to the notion of viscosity solution. this notion will be introduced precisely in Section~\ref{s.hj}. Also note, that according to \cite[Proposition~3.6]{chenmourrat2023cavity}, we have $\nabla f(t,\mfk q) \in \mcl Q(S^D_+) \cap L^2$. This means that the behavior of \eqref{e. HJ smooth} is only governed by the restriction of $\xi$ to $S^D_+$, this is why we can assume only that $\xi$ is convex on $S^D_+$.

In this framework, the Parisi formula can be seen as a consequence of a variational representation for the viscosity solution of \eqref{e. HJ smooth}. According to \cite[Theorem~4.6~(2)]{chen2022hamilton}, for $(t,\mfk q) \in \R_+ \times \mcl Q(S^D_+) \cap L^2$, we have
\begin{equation*}
    f(t,\mfk q)  = \sup_{\mfk p \in \mcl Q(S^D_+) \cap L^\infty} \inf_{\mfk r \in \mcl Q(S^D_+) \cap L^\infty} \left\{ \psi(\mfk q + \mfk p) - \langle \mfk p, \mfk r \rangle_{L^2} + t \int_0^1 \xi( \mfk r(u) ) \d u\right\}.
\end{equation*}
At $\mfk q = 0$ this yields, $\lim_{ N \to +\infty} \bar F_N(t) = \sup_{\mfk p \in \mcl Q(S^D_+)} \mcl P_{t,\xi,P_1}(\mfk p)$, where 
\begin{equation} \label{e. parisi functionnal}
    \mcl P_{t,\xi,P_1}(\mfk p) = \psi(\mfk p) - \sup_{\mfk r \in \mcl Q(S^D_+) \cap L^\infty} \left\{  \langle \mfk p, \mfk r \rangle_{L^2} - t \int_0^1 \xi( \mfk r(u) ) \d u\right\}.
\end{equation}
As explained in the proof of \cite[Proposition~8.1]{chenmourrat2023cavity}, one can recover the “usual” Parisi functional from $\mcl P_{t,\xi,P_1}$ by plugging in paths of the form $\nabla \xi \circ \mfk p$. If we define $\theta(a) = a  \cdot \nabla \xi(a)  - \xi(a)$, we have 
\begin{equation*}
    \mcl P_{t,\xi,P_1}(t \nabla \xi \circ \mfk p) = \psi(t \nabla \xi \circ \mfk p) - t \int_0^1 \theta(\mfk p(u)) \d u.
\end{equation*}

Before moving to the next section, we explain the adjustments that can be made to obtain a simpler Hamilton-Jacobi equation when $\xi$ is assumed to only depend on the diagonal coefficients of its argument. Henceforth, we will refer to models with this additional property as diagonal models. For those models, we identify the function $A \mapsto \xi(A)$ defined on $\R^{D \times D}$ and the function $x \mapsto \xi(\text{diag}(x))$ defined on $\R^D$. Recall that heuristically, the paths $\mfk q$ are to be understood as encoding the limiting distribution of the overlap matrix, $\frac{\sigma \tau^*}{N}$. When the model is diagonal, we do not need to keep track of the full overlap matrix and encoding the distribution of the diagonal coefficients is enough. To this end, let us introduce another space of paths,
\begin{equation*}
    \mcl Q(\R^D_+) = \{ q : [0,1) \to \R_+^D \big| \; \text{$q$ is càdlàg nondecreasing} \}.
\end{equation*}
Here, by $q$ is nondecreasing we mean that for every $u \leq v$, $q(v) - q(u) \in \R^D_+$. Note that $\mcl Q(\R^D_+)$ is isomorphic to the  subset of $\mcl Q(S^D_+)$ composed of the paths that are valued in the set of diagonal matrices. If $q \in \mcl Q(\R^D_+)$ we denote by $\text{diag}(q)$ the associated diagonal matrix valued path. We then define $\mcl Q_\upa(\R^D_+)$ as the set of paths $q \in \mcl Q(\R^D_+)$ such that the path $\text{diag}(q)$ belongs to $\mcl Q_\upa(S^D_+)$. In order to not get mixed up with the previous definition of enriched free energy, we will denote by $\bar F^\text{diag}_N$ the restriction of $\bar F_N$ to $\mcl Q(\R^D_+)$. That is, for every $(t,q) \in \R_+ \times (\mcl Q(\R^D_+) \cap L^1)$,
\begin{equation*}
    \bar F^\text{diag}_N(t,q) = \bar F_N(t,\text{diag}(q)).
\end{equation*}
As previously, the sequence $(\bar F^\text{diag}_N(t,q))_N$ converges to $f(t,\text{diag}(q))$. We let $f^\text{diag}(t,q) =f(t,\text{diag}(q))$, the function $f^\text{diag}$ is Gateaux differentiable on $(0,+\infty) \times \left(\mcl Q_\upa(\R^D_+) \cap L^\infty \right)$ and solves 
\begin{equation} \label{e. HJ smooth diag}
    \begin{cases}
        \partial_t f^\text{diag} - \int \xi(\nabla f^\text{diag}) = 0 &\text{ on } (0,+\infty) \times (\mcl Q_\upa(\R^D_+) \cap L^\infty) \\
        f^\text{diag}(0,\cdot) = \psi^\text{diag}  &\text{ on } \mcl Q_\upa(\R^D_+) \cap L^\infty,
    \end{cases}
\end{equation}
where we have defined $\psi^\text{diag}(q)= \psi(\text{diag}(q))$. This can be checked by differentiating $f^\text{diag}$ and using \eqref{e. HJ smooth}. The point of this derivation is that we can now express the limit free energy of diagonal models as the value at $(t,0)$ of the solution of a Hamilton-Jacobi equation on $\mcl Q(\R^D_+)$ rather than $\mcl Q(S^D_+)$. This yields a different variational representation for the limit free energy.

\section{Hamilton-Jacobi equations on closed convex cones} \label{s.hj}

As explained above, Hamilton-Jacobi equations on closed convex cones, such as \eqref{e. HJ smooth} and \eqref{e. HJ smooth diag}, will play an important role in this paper. Here, we summarize some known results on those equations. We will focus on the theory of Hamilton-Jacobi equations on closed convex cones in finite dimensional Hilbert spaces. But, most of the result exposed below remain valid in infinite dimensional Hilbert spaces \cite{chen2022hamilton}. We will use both settings in the following sections.

Let $(\mcl H, \langle \cdot, \cdot \rangle_{\mcl H})$ be a finite dimensional Hilbert space. Let $\mcl C \subset \mcl H$, we say that $\mcl C$ is a closed convex cone when $\mcl C$ is a closed set and for every $s,t \geq 0$, $x,y \in \mcl C$, we have 
\begin{equation*}
    sx +ty \in \mcl C.
\end{equation*}
In what follows, we fix a nonempty closed convex cone $\mcl C \subset \mcl H$. We will assume that the interior of $\mcl C$ is nonempty and that the only vector belonging to $\mcl C$ and $-\mcl C$ simultaneously is $0$. When dealing with infinite dimensional Hilbert spaces like in \cite{chen2022hamilton}, $\mcl C$ is allowed to have empty interior, but for simplicity we do not consider this possibility in this section.

Let $\psi : \mcl C \to \R$ be a Lipschitz function and $H : \mcl H \to \R$ be such that $H \big|_{\mcl C}$ is locally Lipschitz. We are interested in equations of the form 
\begin{equation} \label{e.hj}
    \begin{cases}
        \partial_t v - H(\nabla v) = 0 \text{ on } (0,+\infty) \times \mathring{\mcl C} \\
        v(0,\cdot) = \psi \text{ on } \mcl C.
    \end{cases}
\end{equation}
When $\mcl C = \R^d$, there is already a rich theory for equations like \eqref{e.hj} \cite[Section~10]{evans}. One of the main accomplishments of this theory  is the introduction of a suitable notion of solution. Equation \eqref{e.hj} may not have any differentiable solution \cite[Section~3.3~Example~6]{evans}. But, by introducing the notion of viscosity solutions (see Definition~\ref{d.visco} below), it is possible to guarantee existence and uniqueness of solutions under some mild conditions, provided that we allow non-differentiable functions to solve \eqref{e.hj}.

Most of the theory of viscosity solutions on $\R^d$ can be adapted for Hamilton-Jacobi equations on closed convex cones \cite{chen2023viscosity}. In principle, when $\R^d$ is replaced by $\mcl C$, some boundary conditions should be enforced to guarantee that \eqref{e.hj} admits a unique solution. But this requirement can be bypassed  provided that some monotony conditions on $\psi$ and $H$ hold.

Given $\mcl D \subset \mcl C$, a function $\phi : (0,+\infty) \times \mcl D \to \R$ is said to be differentiable at $(t,x) \in \mcl (0,+\infty) \times \mcl D$ when there exists a unique $(a,p) \in \R \times \mcl H$ such that the following holds as $(s,y) \in (0,+\infty) \times \mcl D$ tends to $(t,x)$,
\begin{equation*}
    \phi(s,y) = \phi(t,x) + (s-t)a + \langle x-y, p \rangle_{\mcl H} + O \left(|s-t|^2+|x-y|^2_{\mcl H} \right).
\end{equation*}
In this case, $(a,p)$ is denoted $(\partial_t \phi(t,x), \nabla \phi(t,x))$. When $\phi$ is differentiable at every point in $(0,+\infty) \times \mcl D$ and the function $(t,x) \mapsto (\partial_t \phi(t,x), \nabla \phi(t,x))$ is continuous, we say that $\phi$ is smooth. Note that the notions of differentiability and smoothness defined here make sense even when $x$ does not belong to the interior of $\mcl C$ and does not requires $\mcl D$ to be an open set. 

\begin{definition}[Viscosity solutions] \label{d.visco}
    \leavevmode
    \begin{enumerate}
    
        \item An upper semi-continuous function $v : \R_+ \times \mcl C \to \R$ is a viscosity subsolution of \eqref{e.hj} when for every $(t,x) \in (0,+\infty) \times \mathring{\mcl C}$ and every smooth function $\phi : (0,+\infty) \times \mathring{\mcl C} \to \R$ such that $v- \phi$ has a local maximum at $(t,x)$, we have 
        \begin{equation*}
            \partial_t \phi(t,x) - H(\nabla \phi(t,x)) \leq 0.
        \end{equation*}

        \item A lower semi-continuous function $v : \R_+ \times \mcl C \to \R$ is a viscosity supersolution of \eqref{e.hj} when for every $(t,x) \in (0,+\infty) \times \mathring{\mcl C}$ and every smooth function $\phi : \R_+ \times \mathring{\mcl C} \to \R$ such that $v- \phi$ has a local minimum at $(t,x)$, we have 
        \begin{equation*}
            \partial_t \phi(t,x) - H(\nabla \phi(t,x)) \geq 0.
        \end{equation*}
        \item A continuous function $v : \R_+ \times \mcl C \to \R$ is a viscosity solution when it is both a viscosity subsolution and a viscosity supersolution.
    \end{enumerate}
\end{definition}

We define the dual cone $\mcl C^*$ of the cone $\mcl C$ by setting,
\begin{equation} \label{e. def dual cone}
    \mcl C^* = \{ x \in \mcl H \big| \; \forall y \in \mcl C, \; \langle x,y \rangle_\mcl H \geq 0 \}.
\end{equation}
The dual cone $\mcl C^*$ is a closed convex cone and satisfies $(\mcl C^*)^* = \mcl C$. Given $\mcl A, \mcl B \subset \mcl H$, we say that $g : \mcl A \to \R$ is $\mcl B$-nondecreasing when for every $x,x' \in \mcl A$, if $x-x' \in \mcl B$ then $g(x) \geq g(x')$. Note that, if $h : \mcl C \to \R$ is differentiable, then $h$ is $\mcl C^*$-nondecreasing if and only if $\nabla h(\mcl C) \subset \mcl C$.  We define $\mcl V(\mcl C)$ to be the set of functions $v : \mathbb{R}_+ \times \mcl C \to \R$, such that for every $t \geq 0$, $v(t,\cdot)$ is $\mcl C^*$-nondecreasing and such that the following estimates hold
\begin{equation}
    \sup_{\substack{t > 0 \\ x \in \mcl C}} \frac{|v(0,x)-v(t,x)|}{t} < +\infty \text{ and } \sup_{t > 0} \| v(t,\cdot)\|_{\text{Lip}}< +\infty.
\end{equation}
Here, $\| v(t,\cdot)\|_{\text{Lip}}$ denotes the optimal Lipschitz constant of $v(t,\cdot)$. The following theorem is extracted from \cite[Theorem~1.2]{chen2023viscosity}.

\begin{theorem}[\cite{chen2023viscosity}] \label{t. hj cone well posed}
    Let $\psi : \mcl C \to \R$ be a Lipschitz and $\mathcal{C}^*$-nondecreasing function and let $H : \mathcal{H} \to \R$ be such that $H \big|_{\mcl C}$ is $\mcl C^*$-nondecreasing and locally Lipschitz. Then, the Hamilton-Jacobi equation 
\begin{equation} \label{_.hj}
    \begin{cases}
        \partial_t v - H(\nabla v) = 0 \text{ on } (0,+\infty) \times \mathring{\mcl C} \\
        v(0,\cdot) = \psi \text{ on } \mcl C,
    \end{cases}
\end{equation}
admits a unique viscosity solution in $\mcl V(\mcl C)$ that will be denoted $v$.
\end{theorem}

Theorem~\ref{t. hj cone well posed} ensures that \eqref{e.hj} is well posed as long as $\psi$ and $H$ are $\mcl C^*$-nondecreasing. When $H$ is convex and some additional requirements are put on the cone $\mcl C$ it is possible to obtain an explicit representation for the unique viscosity solution. We refer to this representation as the Hopf-Lax representation of the viscosity solution.

Given $g : \mcl C \to \R \cup \{+\infty\}$, we define the (monotone) convex conjugate of $g$ over $\mcl C$ by
\begin{equation} \label{e.convex conjugate}
    g^*(y) = \sup_{x \in \mcl C} \left\{ \langle x,y \rangle_{\mcl H} - g(x) \right\}.
\end{equation}
When $\mcl C = \R^d$, the functions that satisfy $g^{**} = g$ are exactly the lower semicontinuous convex functions \cite[Section~12]{rockafellar1970convex}. If we let 
\begin{equation*}
    g^{**}(x) = \sup_{y \in \mcl C^*} \left\{ \langle x,y \rangle_{\mcl H} - g(y) \right\}
\end{equation*}
denote the convex conjugate of $g^*$ over $\mcl C^*$, we have that $g^{**}$ is lower semi-continuous, convex and $\mcl C$-nondecreasing. So, the functions satisfying $g = g^{**} $ must be lower semicontinuous, convex and $\mcl C$-nondecreasing. 

\begin{definition}
    We say that a closed convex cone $\mcl C$ has the Fenchel-Moreau property when for every $g : \mcl C \to \R \cup\{+\infty\}$ not identically equal to $+\infty$, we have $g^{**} = g$ if and only if $g$ is lower semicontinuous, convex and $\mcl C$-nondecreasing.
\end{definition}

\begin{theorem} [\cite{chen2023viscosity}]
    Assume that $\mcl C$ has the Fenchel-Moreau property. Let $\psi : \mcl C \to \R$ be a Lipschitz and $\mathcal{C}^*$-nondecreasing function and let $H : \mathcal{H} \to \R$ be such that $H \big|_{\mcl C}$ is $\mcl C^*$-nondecreasing, locally Lipschitz and \emph{convex}. Then, the unique viscosity solution $v$ of \eqref{_.hj} admits the Hopf-Lax representation, that is for every $(t,x) \in \R_+ \times \mcl C$,
    \begin{equation*}
        v(t,x) = \sup_{y \in \mcl C} \inf_{z \in \mcl C} \left\{ \psi(x+y) - \langle y,z \rangle_{\mcl H}+ tH(z) \right\}.
    \end{equation*}
\end{theorem}

Most of the cones that we will encounter in this document have the Fenchel-Moreau property. Examples of cones with the Fenchel-Moreau property include $\R^D_+$ and $S^D_+$. We refer to \cite{chen2022fenchelmoreau} for an in-depth study of Fenchel-Moreau cones. Note that the cones $\mcl Q(S^D_+)$, $\mcl Q(\R^D_+)$ and $\mcl Q^j(S^D_+)$, $\mcl Q^j(\R^D_+)$ defined below are all Fenchel-Moreau cones, for those cones the Hopf-Lax representation will therefore be available. Note that when $\psi$ is assumed to be convex and $H$ is possibly nonconvex, another variational representation for the viscosity solution is available \cite[Proposition~6.3]{chen2023viscosity}. Finally, we point out that, as proven in \cite[Theorem~1.2~(1)]{chen2023viscosity}, the comparison principle also remains valid for \eqref{e.hj}. Note that crucially, to use the comparison principle, the condition to be a viscosity solution needs only to be checked on the interior of $\mcl C$. This means that in practice when dealing with solutions of \eqref{e.hj}, the boundary points of $\mcl C$ do not need to be considered.
\begin{theorem}[\cite{chen2023viscosity}] \label{t.comparison}
    Let $u,v \in \mcl V(\mcl C)$, assume that $u$ is a viscosity subsolution and $v$ is a viscosity supersolution. Then,
    \begin{equation*}
        \sup_{(t,x) \in \R_+ \times \mcl C} \left\{u(t,x)-v(t,x) \right\} = \sup_{x \in \mcl C} \left\{ u(0,x) - v(0,x) \right\}.
    \end{equation*}
\end{theorem}
We conclude this section by using the comparison principle to prove a stability result for approximate strong solutions of \eqref{e.hj} on $\mathring{\mcl C}$.
\begin{proposition} \label{p. approximate strong solutions are approximate visco solutions}
    Let $\psi : \mcl C \to \R$ be a Lipschitz and $\mathcal{C}^*$-nondecreasing function, and let $H : \mathcal{H} \to \R$ be such that $H \big|_{\mcl C}$ is $\mcl C^*$-nondecreasing and locally Lipschitz.  Let $g \in \mcl V(\mcl C)$, assume that $g(0,\cdot) = \psi$, that $g$ is differentiable on $(0,+\infty) \times \mathring{\mcl C}$ and that there exists a constant $c >0$ such that
    \begin{equation}
        \sup_{t > 0, x \in \mathring{\mcl C}} |\partial_t g(t,x) - H(\nabla g(t,x))| \leq c.
    \end{equation}
    Then, for every $t \geq 0$, $\sup_{x \in \mcl C} |v(t,x)-g(t,x)| \leq ct$. Where $v$ is the unique viscosity solution of \eqref{_.hj}.
\end{proposition}
\begin{proof}
   Define $\tilde{g}(t,x) = g(t,x) +ct$, we have $\Tilde{g} \in \mcl V(\mcl C)$ and $\Tilde{g}(0,\cdot) = \psi$. In addition, $\Tilde{g}$ is differentiable on $(0,+\infty) \times \mathring{\mcl C}$ and 
   \begin{equation*}
       \partial_t \Tilde{g}  - H(\nabla \Tilde{g}) \geq 0 \text{ on } (0,+\infty) \times \mathring{\mcl C}.
   \end{equation*}
   Therefore $\Tilde{g}$ is supersolution of \eqref{_.hj} on $ \mathring{\mcl C}$. According to the comparison principle \cite[Theorem~1.2~(1)]{chen2023viscosity}, we have $v-\Tilde{g} \leq 0$ on $\R_+ \times \mcl C$. So for every $x \in \mcl C$,
   \begin{equation*}
       v(t,\cdot)-g(t,\cdot) \leq ct.
   \end{equation*}
    Similarly, we show that $(t,x) \mapsto g(t,x)-ct$ is a subsolution and use the comparison principle to deduce $g(t,\cdot)-v(t,\cdot) \leq ct$ on $\R_+ \times \mcl C$. Combining those two bounds, we obtain the desired result. 
\end{proof}

\section{Permutation-invariant objects} \label{s.generalities}

In this section, we introduce basic notations and results for matrix valued paths. Those results will be useful in later sections for the analysis of Hamilton-Jacobi equations.

\subsection{Permutation-invariant matrices} 

Let $m \in \R^{D \times D}$, for every $s \in \msc S_D$ we define $m^s = (m_{s(d),s(d')})_{1 \leq d,d'\leq D}$. When for every $s \in \msc S_D$, $m^s = m$ we say that $m$ is permutation-invariant. An in depth study of permutation-invariant matrices is conducted in \cite{oneill2021doubleconstant}, for the sake of completeness we extract the following result. Recall that $\text{id}_D$ denotes the $D \times D$ identity matrix and $\mathbf{1}_D$ denotes the $D \times D$ matrix whose coefficients are all equal to $1$.

\begin{proposition}[\cite{oneill2021doubleconstant}] \label{p.permutation invariant matrix reduction}
    Let $m \in \R^{D \times D}$ be permutation-invariant. We have
    \begin{equation}
        m = \begin{pmatrix}
             a & t & \cdots & \cdots & t \\ 
             t & \ddots & \ddots & & \vdots \\
             \vdots & \ddots &\ddots & \ddots & \vdots \\
            \vdots & & \ddots & \ddots & t \\
             t & \cdots & \cdots & t & a\\
            \end{pmatrix} = a \text{id}_D + t \left(\mathbf{1}_D - \text{id}_D \right),
    \end{equation}
    where $a = m_{11}$ and $t = m_{12}$. The matrix $m$ admits two eigenvalues $\lambda_1 = a-t$ and $\lambda_2 = a-t +Dt$ (or $1$ eigenvalue when $t = 0$) and can be expressed in terms of its eigenvalues via
    \begin{equation}
        \begin{split}
            m &= \frac{1}{D}\begin{pmatrix}
             \lambda_2 + (D-1)\lambda_1 & \lambda_2-\lambda_1 & \cdots & \cdots & \lambda_2-\lambda_1 \\ 
             \lambda_2-\lambda_1 & \ddots & \ddots & & \vdots \\
             \vdots & \ddots &\ddots & \ddots & \vdots \\
            \vdots & & \ddots & \ddots & \lambda_2-\lambda_1 \\
             \lambda_2-\lambda_1 & \cdots & \cdots &\lambda_2-\lambda_1  & \lambda_2 + (D-1)\lambda_1\\
            \end{pmatrix} \\
            &= \lambda_1 \left(\text{id}_D - \frac{\mathbf{1}_D}{D} \right) + \lambda_2 \frac{\mathbf{1}_D}{D}.
        \end{split}
    \end{equation}
\end{proposition}

\begin{proof}
    Let $m \in \R^{D \times D}$ be permutation-invariant, and let  $a = m_{11}$ and $t = m_{12}$. Let $d,d' \in \{1,\dots,D\}$, if $d \neq d'$ there exists $s \in \msc S_D$, such that $(s(d),s(d')) = (1,2)$, since $m$ is permutation-invariant, this yields $m_{d,d'} = t$. Otherwise, $d = d'$ and there exists $s \in \msc S_D$, such that $(s(d),s(d')) = (1,1)$, from which it follows that $m_{d,d} = a$. This proves that $m = a I_D + t \left(\mathbf{1}_D - \text{id}_D \right)$. Moreover, the characteristic polynomial of the matrix $\mathbf{1}_D$ is $\chi_{\mathbf{1}_D}(X) = (X-D)X^{D-1}$. The eigenvalues of $m$ can be deduced from the fact that if $t \neq 0$, $\chi_m(X) = t^D\chi_{\mathbf{1}_D} \left( \frac{X-a+t}{t}\right)$. 
\end{proof}

In what follows, for every $\lambda_1,\lambda_2 \in \R$, we will denote by $m(\lambda_1,\lambda_2)$ the matrix $\lambda_1 \left(\text{id}_D - \frac{\mathbf{1}_D}{D} \right) + \lambda_2 \frac{\mathbf{1}_D}{D}$. According to Proposition~\ref{p.permutation invariant matrix reduction}, $m(\lambda_1,\lambda_2)$ is the only $D \times D$ permutation-invariant matrix whose eigenvalues are $\lambda_1$ and $\lambda_2$. Recall that $\R^{D \times D}$ is equipped with the Frobenius inner product
\begin{equation*}
    m \cdot n = \sum_{d,d'= 1}^D m_{dd'}n_{dd'}.
\end{equation*}
In particular, we have $m(\lambda) \cdot m(\mu)  = (D-1) \lambda_1 \mu_1 + \lambda_2 \mu_2$. As per \eqref{e.convex conjugate}, we let $\xi^*$ denote the convex conjugate of $\xi$ with respect to the cone $S^D_+$, that is 
\begin{equation*}
    \xi^*(m) = \sup_{n \in S^D_+} \left\{m \cdot n - \xi(n) \right\}.
\end{equation*}
For every $\lambda_1,\lambda_2 \in \R$, we define
\begin{equation} \label{e. def xi_perp}
    \xi_\perp(\lambda_1,\lambda_2) = \xi\left(m\left(\frac{\lambda_1}{D-1},\lambda_2\right)\right).
\end{equation}
We equip $\R^D$ with the standard inner product $x \cdot y = \sum_{d = 1}^D x_d y_d$, and let $\xi^*_\perp$ denote the convex conjugate of $\xi_\perp$ with respect to $\R_+^2$,
\begin{equation*}
    \xi_\perp^*(\lambda) = \sup_{\mu \in \R^2_+} \left\{ \lambda \cdot \mu - \xi_\perp(\mu) \right\}.
\end{equation*}

\begin{proposition} \label{p.convex dual xi_perp}
    For every $\lambda_1,\lambda_2 \in \R$, we have 
    \begin{equation*}
         \xi_\perp^*(\lambda_1,\lambda_2) =\xi^*\left(m\left(\lambda_1,\lambda_2\right)\right) .
    \end{equation*}
\end{proposition}

\begin{proof}
    Clearly, $\xi^*(m(\lambda_1,\lambda_2)) \geq (\xi_\perp)^*(\lambda_1,\lambda_2)$. Conversely, given $n \in S^D_+$, we define $n_0 = \frac{1}{D!} \sum_{s \in \msc S_D} n^s$. The matrix $n_0$ is permutation-invariant and positive semi-definite, by Proposition~\ref{p.permutation invariant matrix reduction}, there exists $\mu_1,\mu_2 \in \R_+^2$ such that $n_0 = m\left(\mu_1/(D-1),\mu_2\right)$. Furthermore, since $\xi$ is convex and permutation-invariant, we have 
    \begin{equation*}
        \xi(n_0) \leq \frac{1}{D!} \sum_{s \in \msc S_D} \xi(n^s) = \xi(n).
    \end{equation*}
    Therefore, 
    \begin{align*}
        m\left(\lambda_1,\lambda_2\right) \cdot n - \xi(n) &=  m\left(\lambda_1,\lambda_2\right) \cdot n_0 - \xi(n) \\
                                                           &= \lambda \cdot \mu - \xi(n) \\
                                                           &\leq \lambda \cdot \mu - \xi\left(m\left(\frac{\mu_1}{D-1},\mu_2\right)\right) \\
                                                           &\leq (\xi_\perp)^*(\lambda_1,\lambda_2).
    \end{align*}
    Taking the supremum over $n \in S^D_+$, we obtain the desired inequality.
\end{proof}

\subsection{Permutation-invariant paths}

Let $V = S^D$ or $V = \R^D$, we equip $V$ with any norm $|\cdot|$. For $p \in [1,+\infty)$, we denote by $L^p([0,1),V)$ or simply $L^p$ the set of functions $h :  [0,1) \to V$ such that $|h|^p$ is integrable. As usual, functions in $L^p$ are to be understood as equivalence classes of functions modulo equality almost everywhere. We define the $p$-norm on $L^p$, by 
\begin{equation*}
    |h|_{L^p} = \left(\int_0^1 |h(u)|^p \d u\right)^{1/p}.
\end{equation*}
The set $L^\infty$ is the set of essentially bounded functions, and we equip it with the sup norm 
\begin{equation*}
    |h|_{L^\infty} = \text{ess-sup}_{u \in [0,1)} |h(u)|.
\end{equation*}
When $V$ is equipped with a scalar product $x \cdot y$ and $|\cdot|$ is the associated Euclidean norm, the norm $|\cdot|_{L^2}$ is Hilbertian and comes from the scalar product
\begin{equation*}
    \langle h,k \rangle_{L^2} = \int_0^1 h(u)\cdot k(u) \d u.
\end{equation*}
At fixed $p \in [1,+\infty]$, the different norms $|\cdot|_{L^p}$ obtained by changing the norm $|\cdot|$ on $V$ are all equivalent. Therefore, the statement $h \in L^p$ does not depend on the particular norm $|\cdot|$ we choose, but the precise value of $|h|_{L^p}$ does. Given $q \in \mcl Q(\R^2_+)$, we let $q^\perp \in \mcl Q(S^D_+)$ be the path defined by $q^\perp(u) = m(q(u))$. For $\mathfrak q \in \mcl Q(S^D_+)$ and $s \in \msc S_D$, we define $\mathfrak q^s \in \mcl Q(S^D_+)$ by $\mathfrak q^s(u) = (\mathfrak q(u))^s = (\mfk q_{s(d),s(d')}(u))_{1 \leq d,d'\leq D}$. As previously, we equip $S^D$ with the Frobenius inner product and $\R^2$ with the standard inner product.

\begin{proposition} \label{p.permutation invariant paths}
    Let $\mfk q \in \mcl Q(S^D_+)$, assume that for every $s \in \msc S_D$, we have $\mfk q^s = \mfk q$, then there exists $q =(q_1,q_2) \in \mcl Q(\R^2_+)$ such that $\mfk q = q^\perp$. In addition, for every $p \in [1,+\infty]$, $\mfk q  \in L^p$ if and only if $q \in L^p$. Moreover, when $\mfk q \in L^2$, we have for every $q' = (q'_1,q'_2) \in \mcl Q(\R^2_+) \cap L^2$,
    \begin{equation*}
        \langle \mfk q, (q')^\perp \rangle_{L^2} =  \langle q_1, (D-1) q'_1 \rangle_{L^2} + \langle q_2, q'_2 \rangle_{L^2}. 
    \end{equation*}
\end{proposition}

\begin{proof}
    Let $u \in [0,1)$, for every $s \in \msc S_D$, the matrix $\mfk q(u) \in S^D_+$ satisfies $(\mfk q(u))^s = \mfk q(u)$. According to Proposition~\ref{p.permutation invariant matrix reduction}, there exists $q(u) \in \R^2_+$ such that $\mfk q(u)  = m(q(u))$. If $u \leq v$, we have
    \begin{equation*}
        m(q(v)-q(u)) = \mfk q(v) - \mfk q(u) \in S^D_+,
    \end{equation*}
    so $q_1(v) -q_1(u) \geq 0$ and $q_2(v) -q_2(u) \geq 0$. 
    
    In addition, the map $m(\lambda_1,\lambda_2) \mapsto (\lambda_1,\lambda_2)$ is continuous according to Proposition~\ref{p.permutation invariant matrix reduction}. This justifies that $q \in \mcl Q(\R^2_+)$ and by definition $q^\perp = \mfk q$. Let $|x|_p = \left(\frac{1}{D} \sum_{d=1}^D |x_d|^p \right)^{1/p}$ denote the normalized $p$-norm on $\R^D$ and $|x|_{\infty} = \sup_{1 \leq d \leq D} |x_d|$. We equip $S^D$ with the norm
    \begin{equation*}
        N_2(m) = \sup_{|x|_2 = 1} |mx|_2.
    \end{equation*}
    When $m \in S^D_+$, $N_2(m)$ is equal to the biggest eigenvalue of $m$. Therefore, for every $\lambda_1,\lambda_2 \in \R_+$,
    \begin{equation*}
        N_2(m(\lambda_1,\lambda_2)) = \max\{\lambda_1,\lambda_2\}.
    \end{equation*}
    Thus, for every $u \in [0,1)$, $N_2(\mfk q(u)) = |q(u)|_\infty$, in particular $\mfk q \in L^p$ if and only if $q \in L^p$. 

    Finally, assume that $\mfk q \in L^2$, then by the previous argument $q \in L^2$. Recall that we have $m(\lambda) \cdot m(\mu) = (D-1)\lambda_1 \mu_1 + \lambda_2 \mu_2$, for every $q' \in \R^2_+$, we have 
    \begin{align*}
        \langle \mfk q,(q')^\perp \rangle_{L^2} &= \int_0^1 m(q(u))\cdot m(q'(u)) \d u \\
                                                &= \int_0^1 (D-1)q_1(u)q'_1(u) + q_2(u)q'_2(u) \d u \\
                                                &=  \langle q_1, (D-1) q'_1 \rangle_{L^2} + \langle q_2, q'_2  \rangle_{L^2}.
    \end{align*}
\end{proof}

\subsection{Piecewise linear approximations of paths} \label{ss. linear approximation}

Given a finite dimensional vector space $V$ and a cone $\mathcal{C} \subset V$. For every $j \geq 1$, we define 
\begin{equation} \label{e.approx cone}
    \mcl Q^j(\mathcal{C}) = \{ x \in \mathcal{C}^j \big| \; \forall i \in \{1,\dots,j\}, \; x_i - x_{i-1} \in \mathcal{C} \}.
\end{equation}
As previously, we only use $(V,\mcl C) = (S^D, S^D_+)$ or $(V,\mcl C) = (\R^D, \R^D_+)$. We will always adopt the convention $x_0 = 0$. Given $x \in \mcl Q^j(\R^D_+)$, we define $\Lambda^j x \in \mcl Q(\R^D_+)$ to be the path that linearly interpolates between the values $(0,0),(\frac{1}{j},x_1),(\frac{2}{j},x_2),\dots,(1,x_j)$. More precisely, for every $u \in [0,1)$,
\begin{equation}
    \Lambda^j x(u) = \sum_{i =1}^j \mathbf{1}_{\left[ \frac{i-1}{j},\frac{i}{j} \right)}(u) \left( x_{i-1}+j \left(u - \frac{i-1}{j} \right)(x_i -x_{i-1} )\right).
\end{equation}
We will write $\mathbf{1}_i$ instead of $\mathbf{1}_{\left[ \frac{i-1}{j},\frac{i}{j} \right)}$ in what follows, with the understanding that $\mathbf{1}_{j+1} = 0$. We also define 
\begin{equation*}
    \tilde{\mathbf{1}}_i(u) = j \left(u - \frac{i-1}{j} \right) \mathbf{1}_i(u) + j \left(\frac{i+1}{j} - u \right) \mathbf{1}_{i+1}(u).
\end{equation*}
Note that we have $\Lambda^j x= \sum_{i = 1}^j x_i \tilde{\mathbf{1}}_i$. 

We equip $\R^2$ with the normalized $\ell^1$-norm, $|v| = \frac{1}{2}(|v_1|+|v_2|)$ and the standard normalized scalar product $v \cdot w= \frac{1}{2}(v_1w_1 +v_2w_2)$. For every $p \in [1,+\infty)$, we equip $(\R^2)^j$ with the norm $|\cdot|_p$, defined by
    \begin{equation}
            |x|_p = \left( \frac{1}{j}\sum_{i = 1}^j |x_i|^p \right)^{\frac{1}{p}}.
    \end{equation}
    We also equip $(\R^2)^j$ with the normalized scalar product,
    \begin{equation*}
        \langle x,y \rangle_j = \frac{1}{j} \sum_{i=1}^j x_i \cdot y_i.
    \end{equation*}
    Given a path, $p \in \mcl Q(\R_+) \cap L^1$ we set 
    \begin{equation*}
        \Lambda_j p(u) = (  \langle p(u), j \tilde{\mathbf{1}}_i \rangle_{L^2})_{1 \leq i \leq j}, 
    \end{equation*}
    this defines a path $\Lambda_j p \in \mcl Q^j(\R_+)$. For every $q \in \mcl Q(\R_+^D)$, if we write $q(u) = (q_d(u))_{1 \leq d \leq D}$, then $q_d \in \mcl Q(\R_+)$ and we define 
    \begin{equation*}
        \Lambda_j q (u) =(\Lambda_j q_1(u),\dots,\Lambda_j q_D(u)).
    \end{equation*}
    The linear maps $\Lambda^j : \mcl Q^j(\R^2_+) \to \mcl Q(\R^2_+) \cap L^1$ and $\Lambda_j :  \mcl Q(\R^2_+) \cap L^1 \to  \mcl Q^j(\R^2_+)$ form an adjoint pair in the following sense, for every $p \in \mcl Q(\R^2_+) \cap L^1$ and $x \in \mcl Q^j(\R^2_+)$, we have 
    \begin{equation*}
        \langle \Lambda^j x , p \rangle_{L^2} =  \langle  x , \Lambda_j p \rangle_j.
    \end{equation*}

\begin{proposition} \label{p. path approximation}
   There exists a constant $c > 0$ such that for every $j \geq 1$ and every $q = (q_1,q_2) \in \mcl Q(\R^2_+) \cap L^\infty$, we have 
    \begin{equation*}
        |q - \Lambda^j \Lambda_j q|_{L^1} \leq \frac{c|q|_{L^\infty}}{j}.
    \end{equation*}
\end{proposition}

\begin{proof}
     Let $p \in \mcl Q(\R_+) \cap L^\infty$, to lighten notations, we let $p_i = p \left( \frac{i}{j}\right)$ and define $\overline{p} = (p_i)_{1 \leq i \leq j} \in \mcl Q^j(\R_+)$. For every $i \in \{1,\dots,j\}$, we have 
    \begin{equation*}
             \frac{p_i+p_{i-1}}{2} \leq  \langle p, j \tilde{\mathbf{1}}_i \rangle_{L^2} \leq \frac{p_i+p_{i+1}}{2}.
    \end{equation*} 
        Therefore $|j p \cdot \Tilde{\mathbf{1}_i} - p_i| \leq \max \left\{ \frac{p_{i+1}-p_{i}}{2} ,\frac{p_i-p_{i-1}}{2}\right\} \leq \frac{p_{i+1}-p_{i-1}}{2}$. Summing over $i$, we obtain $ |\Lambda_j p-\overline{p}|_1 \leq \frac{|p|_{L^\infty}}{j}$. In addition, given $x \in \mcl (\R_+)^j$, we have
    \begin{equation*}
        |\Lambda^j(x)|_{L^1} \leq \sum_{i=1}^j |x_i| |\tilde{\mathbf{1}_i}|_{L^1}  = |x|_1.
    \end{equation*}
    Therefore,
    \begin{equation*}
        |\Lambda^j \Lambda_j p - \Lambda^j(\overline{p})|_{L^1} \leq |\Lambda_j p-\overline{p}|_1 \leq \frac{|p|_{L^\infty}}{j}.
    \end{equation*}
    Observe, that for $u \in \left[\frac{i-1}{j},\frac{i}{j} \right)$, $|\Lambda^j\overline{p}(u)-p(u)| \leq p_i -p_{i-1}$, thus 
    \begin{equation*}
        |\Lambda^j(\overline{p})-p|_{L^1} \leq \frac{|p|_{L^\infty}}{j}.
    \end{equation*}
    Combining the previous two displays, we obtain 
    \begin{equation*}
        |\Lambda^j \Lambda_j p - p|_1 \leq  |\Lambda^j \Lambda_j p - \Lambda^j(\overline{p})|_1 + |\Lambda^j(\overline{p})-p|_1 \leq \frac{|p|_{L^\infty}}{j} + \frac{|p|_{L^\infty}}{j} = \frac{2|p|_{L^\infty}}{j}.
    \end{equation*}
    Finally, given $q \in \mcl Q(\R_+^2) \cap L^\infty$ if we write $q(u) = (q_1(u),q_2(u))$, then $q_1,q_2 \in \mcl Q(\R_+) \cap L^\infty$. Applying the bound in the previous display to $q_1$ and $q_2$, we obtain
    \begin{align*}
        |\Lambda^j\Lambda_j q - q|_{L^1} &= \frac{1}{2} \left(|\Lambda^j\Lambda_j q_1 - q_1|_{L^1} + |\Lambda^j\Lambda_j q_2 - q_2|_{L^1} \right) \\
                                        &\leq \frac{|q_1|_{L^\infty}+|q_2|_{L^\infty}}{j} \\
                                        &\leq \frac{2|q|_{L^\infty}}{j}.
    \end{align*}
\end{proof}

Recall that given a closed convex cone $\mcl C$ we have defined its dual cone $\mcl C^*$ in \eqref{e. def dual cone}.

\begin{proposition} \label{p.properties of Lambda}
    For every $x \in (\R^2_+)^j$, we have
    \begin{equation*}
        |\Lambda^j(x)|_{L^1} \leq |x|_1.
    \end{equation*}
    In addition, if $x \in (\mcl Q^j(\R^2_+))^*$, then $(\Lambda^j(x))^\perp \in (\mcl Q(S^D_+) \cap L^2)^*$.
\end{proposition}

\begin{proof}

    Let $x =(x_1,x_2) \in (\R^2_+)^j$, we have $|\tilde{ \mathbf{1}_i}|_{L^1} = \frac{1}{j}$. So, 
    \begin{equation*}
        |\Lambda^j x_1|_{L^1}= \left| \sum_{i = 1}^j x_{1i} \tilde{ \mathbf{1}_i} \right|_{L^1} \leq \sum_{i = 1}^j x_{1i} \left|  \tilde{ \mathbf{1}_i} \right|_{L^1} = \frac{1}{j} \sum_{i =1}^j |x_{1i}|.
    \end{equation*}
    It follows that,
    \begin{equation*}
        |\Lambda^j x|_{L^1} =  \frac{|\Lambda^j x_1|_{L^1} + |\Lambda^j x_2|_{L^1}}{2} \leq \frac{|x_1|_1 + |x_2|_1}{2} = |x|_1. 
    \end{equation*}
    This proves the first part of the proposition, let us now further assume that $x \in (\mcl Q^j(\R^2_+))^*$. Let $\mfk q \in \mcl Q(S^D_+) \cap L^2$,  the path $\frac{1}{D!} \sum_{s \in \msc S_D} \mfk q^s$ is permutation-invariant. According to Proposition~\ref{p.permutation invariant paths}, there exists $q \in \mcl Q(\R^2_+) \cap L^2$ such that $\frac{1}{D!} \sum_{s \in \msc S_D} \mfk q^s = q^\perp$. We have
    \begin{align*}
        \langle \mfk q, (\Lambda^jx)^\perp \rangle_{L^2} &= \left\langle \mfk q,\frac{1}{D!} \sum_{s \in \msc S_D} (\Lambda^jx)^{\perp,s^{-1}}\right\rangle_{L^2} \\
                                                           &= \left\langle \frac{1}{D!} \sum_{s \in \msc S_D} \mfk q^s,(\Lambda^jx)^\perp \right\rangle_{L^2} \\
                                                           &= \left\langle q^\perp, (\Lambda^jx)^\perp \right\rangle_{L^2} \\
                                                           &= \left\langle (D-1)q_1, \Lambda^jx_1 \right\rangle_{L^2} + \left\langle q_2, \Lambda^jx_2 \right\rangle_{L^2} \\
                                                           &= \left\langle (D-1)\Lambda_j q_1, x_1 \right\rangle_j + \left\langle \Lambda_jq_2, x_2 \right\rangle_{j}.
    \end{align*}
    We have $((D-1)\Lambda_j q_1,\Lambda_j q_2) \in \mcl Q^j(\R^2_+)$, so by definition of $(\mcl Q^j(\R^2_+))^*$, the last line in the previous display is $\geq 0$. This justifies $(\Lambda^j(x))^\perp \in (\mcl Q(S^D_+) \cap L^2)^*$.
\end{proof}

Recall the definition of $\xi_\perp$ from \eqref{e. def xi_perp}. Also recall that, given $r \in \mcl Q$ we use the notation $\int h(r)$ as a shorthand for $\int_0^1 h(r(u)) \d u$.

\begin{proposition} \label{p. apprxoimation decrease H}
    Let $q \in \mcl Q(\R^2_+) \cap L^1$, for every $j \geq 1$, we have 
    \begin{equation*}
        \int \xi_\perp(\Lambda^j \Lambda_j q) \leq \int \xi_\perp(q).
    \end{equation*}
\end{proposition}

\begin{proof}
    We let $\rho_1,\dots,\rho_j \in \R^2_+$ denote the coordinates of $\Lambda_j q$. By convexity of $\xi$, we have
    \begin{align*}
        \int \xi_\perp(\Lambda^j \Lambda_j q) &= \int_0^1 \xi_\perp(\Lambda^j \Lambda_j q(u)) \d u \\
                                              &= \sum_{i = 1}^j \int_{\left[\frac{i-1}{j},\frac{i}{j} \right)}\xi_\perp\left( \rho_{i-1} + j \left(u - \frac{i-1}{j} \right)(\rho_{i} - \rho_{i-1})\right) \d u\\
                                              &= \sum_{i = 1}^j \int_{\left[\frac{i-1}{j},\frac{i}{j} \right)}\xi_\perp\left( j\left(\frac{i}{j}-u\right)\rho_{i-1} + j \left(u - \frac{i-1}{j} \right)\rho_{i} \right) \d u\\
                                              &\leq \sum_{i = 1}^j \int_{\left[\frac{i-1}{j},\frac{i}{j} \right)} j\left(\frac{i}{j}-u\right)\xi_\perp\left(\rho_{i-1}\right) + j \left(u - \frac{i-1}{j} \right)\xi_\perp\left(\rho_{i} \right) \d u \\
                                              &= \sum_{i = 1}^j \frac{\xi_\perp\left(\rho_{i-1}\right) + \xi_\perp\left(\rho_{i} \right)}{2j}.
    \end{align*}   
    Using Jensen's inequality, it follows that $\xi_\perp\left(\rho_{i} \right) \leq \langle \xi_{\perp}(q), j \Tilde{\mathbf{1}_i}\rangle_{L^2}$. Finally, since $\sum_{i = 1}^j \tilde{\mathbf{1}}_i = 1$, we obtain 
    \begin{equation*}
        \int \xi_\perp(\Lambda^j \Lambda_j q) \leq \int \xi_\perp(q).
    \end{equation*}
\end{proof}

\section{The free energy and Hamilton-Jacobi equations} \label{s. fe and hj}

We recall that $f(t,\mfk q) = \lim_{N \to +\infty} \bar F_N(t,\mfk q)$. In this section, we use the fact that $f$ solves \eqref{e. HJ smooth} and the permutation invariance of the model to show that $\lim \bar F_N(t,0)$ can be expressed as the value at $(t,0)$ of the solution of a reduced Hamilton-Jacobi equation.

\subsection{The equation for models with matrix valued paths} \label{ss.matrix paths}

Given $q \in \mcl Q(\R^2_+)$, recall that we have defined $q^\perp(u) = m(q(u))$. The goal of this section is to show that the functions defined by $f^\perp(t,q) =f(t,q^\perp)$ is the viscosity solution of some Hamilton-Jacobi on $\mcl Q(\R_+^2)$. Since we know that $f$ is Gateaux differentiable and solves $\eqref{e. HJ smooth}$, this will basically only amount to computing the Gateaux derivatives of $f^\perp$.

\begin{proposition} \label{p. fe is perm invariant}
    For every $t \geq 0$, $\mfk q \in \mcl Q(S^D_+) \cap L^2$ and $s \in \msc S_D$, we have 
    \begin{equation}
        f(t,\mfk q^s) = f(t,\mfk q).
    \end{equation} 
\end{proposition}  

\begin{proof}
     Let $\mfk q \in \mcl Q(S^D_+)$ and $t \geq 0$, recall from \eqref{e. enriched fe} that we have
    \begin{equation*}
        F_N(t,\mfk q) = -\frac{1}{N} \E \log \iint \exp(H_N^{t,\mfk q}(\sigma,\alpha)) dP_N(\sigma) d \mfk R(\alpha),
    \end{equation*}
    where
    \begin{equation*}
        H_N^{t,\mfk q}(\sigma,\alpha) = \sqrt{2t} H_N(\sigma) -t N \xi \left(\frac{\sigma \sigma^* }{N}\right) + \sqrt{2} W_N^\mfk q(\alpha) \cdot \sigma - \mfk q(1) \cdot \sigma \sigma^*.
    \end{equation*}
    The process $( H_N^{t,\mfk q}(\sigma,\alpha))_{\sigma \in \R^D, \alpha \in \mfk U}$ is a Gaussian process with the following mean and covariance,
    \begin{align*}
        \E  H_N^{t,\mfk q}(\sigma,\alpha) &= -t N \xi \left(\frac{\sigma \sigma^* }{N}\right)  - \mfk q(1) \cdot \sigma \sigma^*, \\
        \text{Cov} \left(  H_N^{t,\mfk q}(\sigma,\alpha),  H_N^{t,\mfk q}(\tau,\beta) \right) &= 2N \left( t \xi \left( \frac{\sigma \tau^*}{N}\right) + \frac{1}{N} \sum_{i = 1}^N \sigma_i \cdot \mfk q(\alpha \wedge \beta) \tau_i \right).
    \end{align*}
   Let us show that for every $s \in \msc S_D$, we have 
    \begin{equation*}
        (H_N^{t,\mfk q^{s^{-1}}}(\sigma,\alpha))_{\sigma \in \R^D,\alpha \in \mfk U} \overset{(d)}{=} (H_N^{t,\mfk q}(\sigma^{s},\alpha))_{\sigma \in \R^D,\alpha \in \mfk U}.
    \end{equation*}

    To proceed, we compute the covariance and the mean of those two Gaussian processes and discover that they are equal. Let $\sigma, \tau \in \R^D$ and $\alpha,\beta \in \mfk U$, since $\xi$ is permutation-invariant, we have $\xi \left( \frac{(\sigma^s) (\tau^s)^*}{N}\right) = \xi \left( \frac{\sigma \tau^*}{N}\right)$. Therefore,
    \begin{align*}
         \E  H_N^{t,\mfk q^{s^{-1}}}(\sigma,\alpha)  H_N^{t,\mfk q^{s^{-1}}}(\tau,\beta) &= 2N \left( t \xi \left( \frac{\sigma \tau^*}{N}\right) + \frac{1}{N} \sum_{i = 1}^N \sigma_i \cdot \mfk q^{s^{-1}}(\alpha \wedge \beta) \tau_i \right) \\
                                                                       &= 2N \left( t \xi \left( \frac{\sigma \tau^*}{N}\right) + \frac{1}{N} \sum_{i = 1}^N \sigma^s_i \cdot \mfk q(\alpha \wedge \beta) \tau^s_i \right) \\
                                                                       &= 2N \left( t \xi \left( \frac{(\sigma^s) (\tau^s)^*}{N}\right) + \frac{1}{N} \sum_{i = 1}^N \sigma^s_i \cdot \mfk q(\alpha \wedge \beta) \tau^s_i \right) \\
                                                                       &=  \E  H_N^{t,\mfk q}(\sigma^s,\alpha)  H_N^{t,\mfk q}(\tau^s,\beta).
    \end{align*}
    In addition, 
    \begin{align*}
        \E H_N^{t,\mfk q^{s^{-1}}}(\sigma,\alpha) &= - Nt \xi \left( \frac{\sigma \sigma^*}{N}\right)  - \mfk q^{s^{-1}}(1) \cdot \sigma \sigma^* \\
                                                  &= - Nt \xi \left( \left( \frac{\sigma \sigma^*}{N}\right)^s\right)  - \mfk q(1) \cdot \left( \sigma \sigma^*\right)^s \\
                                                  &= - Nt \xi \left( \frac{\sigma^s (\sigma^s)^*}{N}\right)  - \mfk q(1) \cdot \sigma^s (\sigma^s)^* \\
                                                  &=  \E H_N^{t,\mfk q}(\sigma^s,\alpha).
    \end{align*}
    The desired equality in law follows. We now show that $f(t,\mfk q^s) = f(t,\mfk q)$. According to the previous result and the permutation invariance of $P_N$, we have 
   \begin{align*}
       F_N(t,\mfk q^s)  &=  -\frac{1}{N} \E \log \iint \exp(H_N^{t,\mfk q^s}(\sigma,\alpha)) dP_N(\sigma) d \mfk R(\alpha) \\
                        &=  -\frac{1}{N} \E \log \iint \exp(H_N^{t,\mfk q}(\sigma^{s^{-1}},\alpha)) dP_N(\sigma) d \mfk R(\alpha) \\
                        &=  -\frac{1}{N} \E \log \iint \exp(H_N^{t,\mfk q}(\sigma,\alpha)) dP_N(\sigma) d \mfk R(\alpha) \\
                        &= F_N(t,\mfk q).
   \end{align*}
    Letting $N \to +\infty$, we obtain the desired result.
    
\end{proof}
    
\begin{proposition} \label{p.permutation invariant fe}
    For every $(t,q) \in (0,+\infty) \times ( \mcl Q_\upa(\R_+^2) \cap L^\infty)$, $f(t,\cdot)$ is Gateaux differentiable at $q^\perp$ and the path $\nabla f(t,q^\perp) \in \mcl Q(S^D_+)$ is permutation-invariant that is, for every $u \in [0,1)$ and $s \in \msc S_D$,
    \begin{equation*}
        (\nabla f(t,q^\perp)(u))^s = \nabla f(t,q^\perp)(u).
    \end{equation*}
\end{proposition}

\begin{proof}
    According to \cite[Proposition~8.6]{chenmourrat2023cavity}, $f(t,\cdot)$ is Gateaux differentiable at $q^\perp$. Let $\kappa \in L^2([0,1),S^D)$ such that for $\varepsilon > 0$ small enough $q^\perp + \varepsilon\kappa \in \mcl Q(S^D_+)$. According to Proposition~\ref{p. fe is perm invariant}, we have for every $s \in \msc S_D$,
    \begin{equation*}
        \frac{f(t,q^\perp + \varepsilon \kappa^s) - f(t,q^\perp)}{\varepsilon} =  \frac{f(t,q^\perp + \varepsilon \kappa) - f(t,q^\perp)}{\varepsilon}.
    \end{equation*}
    Letting $\varepsilon \to 0$, we obtain that $\langle \nabla f(t,q^\perp) , \kappa \rangle_{L^2} = \langle \nabla f(t,q^\perp) , \kappa^s \rangle_{L^2}$. So,
    \begin{equation*}
        \langle \nabla f(t,q^\perp), \kappa \rangle_{L^2} = \langle (\nabla f(t,q^\perp))^{s^{-1}} , \kappa \rangle_{L^2}.
    \end{equation*}
    This means that $(\nabla f(t,q^\perp))^{s^{-1}} = \nabla f(t,q^\perp)$.    
\end{proof}

For $q \in \mcl Q(\R^2_+) \cap L^2$, we define $f^\perp(t,q) = f(t,q^\perp)$, $\psi^\perp(q) = \psi(q^\perp)$. Recall the definition of $\xi_\perp$ in \eqref{e. def xi_perp}.

\begin{proposition} \label{p.fe is a strong solution}
    The function $f^\perp$ is Gateaux differentiable at every $(t,q) \in (0,+\infty) \times (\mcl Q_\upa(\R^2_+) \cap L^\infty)$, and it satisfies
    \begin{equation} \label{e.HJ matrix}
        \begin{cases}
            \partial_t f^\perp - \int \xi_\perp(\nabla f^\perp) = 0 \text{ on } (0,+\infty) \times (\mcl Q_\upa(\R^2_+) \cap L^\infty) \\
            f^\perp(0,\cdot) = \psi^\perp.
        \end{cases}
    \end{equation}

\end{proposition}

\begin{proof}
        According to \cite[Propositions~7.2~\&~8.6]{chenmourrat2023cavity}, $f$ is Gateaux differentiable at every $(t,\mfk q) \in (0,+\infty) \times (\mcl Q_\upa(S^D_+) \cap L^\infty)$ and we have
    \begin{equation*}
        \partial_t f(t,\mfk q) = \int \xi(\nabla f(t,\mfk q)).
    \end{equation*}
    Let $t > 0$ and $q = (q_1,q_2) \in \mcl Q_\upa(\R^2_+) \cap L^\infty$ and let $\kappa = (\kappa_1,\kappa_2) \in L^2([0,1),\R^2)$ such that for $\varepsilon > 0$ small enough $q +\varepsilon \kappa \in\mcl Q_\upa(\R^2_+)$. Passing to the limit as $\varepsilon \to 0$ in
    \begin{equation*}
        \frac{f^\perp(t,q+\varepsilon \kappa) -f^\perp(t,q) }{\varepsilon} =  \frac{f(t,q^\perp+\varepsilon \kappa^\perp) -f(t,q^\perp) }{\varepsilon},
    \end{equation*}
    yields $\langle \nabla f^\perp(t,q), \kappa \rangle_{L^2} = \langle \nabla f(t,q^\perp),\kappa^\perp \rangle_{L^2}$. According to Proposition~\ref{p.permutation invariant fe}, the path $\nabla f(t,q^\perp)$ is permutation-invariant. According to Proposition~\ref{p.permutation invariant paths}, there exists $r = (r_1,r_2) \in \mcl Q(\R^2_+)$ such that, $\nabla f(t,q^\perp) = r^\perp$ and we have 
    \begin{equation*}
        \langle (\nabla f^\perp(t,q))_1, \kappa_1 \rangle_{L^2} +  \langle (\nabla f^\perp(t,q))_2, \kappa_2 \rangle_{L^2} = \langle (D-1)r_1, \kappa_1 \rangle_{L^2} +  \langle r_2, \kappa_2 \rangle_{L^2}.
    \end{equation*}
    So, $\nabla f^\perp(t,q) = \left((D-1)r_1 ,r_2 \right)$. In particular, for every $u \in [0,1)$,
    \begin{equation*}
        \xi_\perp(\nabla f^\perp(t,q)) = \xi(r^\perp) = \xi(\nabla f(t,q^\perp)).
    \end{equation*}
    Finally, we have 
    \begin{align*}
        \partial_t f^\perp(t,q) = \partial_t f(t,q^\perp) = \int \xi(\nabla f(t,q^\perp)) = \int \xi_\perp(\nabla f^\perp(t,q)).
    \end{align*}
\end{proof}

If \eqref{e.HJ matrix} was written on $\mcl Q(\R^2_+) \cap L^2$ rather than $\mcl Q_\upa(\R^2_+) \cap L^\infty$, then we could immediately conclude that $f^\perp$ is a solution in the viscosity sense and \eqref{e.symmetric optimizer matrix } would directly follow from the Hopf-Lax representation. This is indeed how are going to argue, but to do so we need to show that we can neglect boundary points. That is, paths in $\mcl Q(\R^2_+) \cap L^2$ that do not belong to $\mcl Q_\upa(\R^2_+) \cap L^\infty$. This can be done using the content of Section~\ref{s.hj}. To connect the two settings, we will need to consider finite dimensional approximations of \eqref{e.HJ matrix}. Recall the definition of the lift and projection maps $\Lambda^j$ and $\Lambda_j$ from Section~\ref{ss. linear approximation}. Given $(t,x) \in \R_+ \times \mcl Q^j(\R_+^2)$, we define $f^{\perp,j}(t,x) = f^\perp(t,\Lambda^jx)$ and $H_\perp^j(x) = \int \xi_\perp(\Lambda^j x)$.

\begin{proposition} \label{p.the approximation are approximate strong solutions}
   There exists a constant $c > 0$ such that the following holds. For every $j \geq 1$, the function $f^{\perp,j}$ is differentiable on $(0,+\infty) \times \mathring{\mcl Q^j(\R^2_+)}$ and for every $(t,x) \in (0,+\infty) \times \mathring{\mcl Q^j(\R^2_+)}$, 
   \begin{equation*}
       |\partial_t f^{\perp,j}(t,x) - H_\perp^j(\nabla f^{\perp,j}(t,x))| \leq \frac{c}{j}.
   \end{equation*}
   Furthermore, $f^{\perp,j}  \in \mcl V(\mcl Q^j(\R^2_+))$.
\end{proposition}

\begin{proof}
   
    \noindent \emph{Step 1.} We show that $f^{\perp,j} \in \mcl V(\mcl Q^j(\R^2_+))$.

    According to \cite[Proposition~5.1]{chenmourrat2023cavity}, $f : \R_+ \times (\mcl Q(S^D_+) \cap L^1)\to \R$ is Lipschitz. More precisely, we have for every $(t,\mfk q),(t',\mfk q') \in \R_+ \times (\mcl Q(S^D_+) \cap L^1)$,
    \begin{equation*}
        |f(t,\mfk q) - f(t',\mfk q')| \leq |\mfk q- \mfk q'|_{L^1} +|t-t'| \sup_{|a| \leq 1} |\xi(a)|.
    \end{equation*}
    Let $|\cdot|$ denote the Frobenius norm on $S^D$. There exists a constant $c > 0$ depending only on $|\cdot|$ such that the following holds. For every $x \in \mcl Q^j(\R_+^2)$, we have 
    \begin{align*}
        |(\Lambda^jx)^\perp|_{L^1} &= \int_0^1 |m(\Lambda^jx(u))| \d u \\
                                     &= \int_0^1 |\Lambda^jx_1(u)|\left|\text{id}_D - \frac{\mathbf{1}_D}{D}\right|  + |\Lambda^jx_2(u)| \left|\frac{\mathbf{1}_D}{D}\right| \d u \\
                                     &\leq c\int_0^1 \frac{|\Lambda^jx_1(u)|+ |\Lambda^jx_2(u)|}{2} \d u \\
                                     &= c|\Lambda^jx|_{L^1} \\
                                     &\leq c |x|_1,
    \end{align*}
    where the last line is a consequence of Proposition~\ref{p.properties of Lambda}. Since $f^{\perp,j}(t,x) = f(t,(\Lambda^jx)^\perp)$, it follows that $f^{\perp,j} : \R_+ \times \mcl Q^j(\R^2_+) \to \R$ is Lipschitz. In particular,  
    \begin{equation*}
       \sup_{t > 0,x \in \mcl Q^j(\R^2_+)} \frac{|f^{\perp,j}(t,x)-f^{\perp,j}(0,x)|}{t} < +\infty \text{ and } \sup_{t > 0} \| f^{\perp,j}(t,\cdot)\|_{\text{Lip}}< +\infty.
    \end{equation*}
    Furthermore, according to \cite[Proposition~3.6]{chenmourrat2023cavity}, for every $t \geq 0$, $f(t,\cdot)$ is $(\mcl Q(S^D_+) \cap L^2)^*$-nondecreasing. Let $x,y \in \mcl Q^j(\R^2_+)$, such that $y -x \in (\mcl Q^j(\R^2_+))^*$, we have $(\Lambda^j x)^\perp, (\Lambda^j y)^\perp \in \mcl Q(S^D_+) \cap L^\infty$ and according to Proposition~\ref{p.properties of Lambda}, $(\Lambda^j(y-x))^\perp \in (\mcl Q(S^D_+) \cap L^2)^*$. Therefore,
     \begin{equation*}
         f^{\perp,j}(t,y) - f^{\perp,j}(t,x) = f(t,(\Lambda^j y)^\perp) - f(t,(\Lambda^j x)^\perp) \geq 0.
     \end{equation*}
     Thus, $f^{\perp,j}(t,\cdot)$ is $(\mcl Q^j(\R^2_+))^*$-nondecreasing and we have proven that $f^{\perp,j} \in \mcl V(\mcl Q^j(\R^2_+))$. This concludes Step 1.   
        
    \noindent \emph{Step 2.} We show that there exists $c > 0$ such that for every $j \geq 1$ and  every $(t,x) \in (0,+\infty) \times \mathring{\mcl Q^j(\R_+^2)}$, $f^{\perp,j}$ is differentiable at $(t,x)$ and  
    \begin{equation*}
        |\partial_t f^j(t,x) - H_\perp^j(\nabla f^j(t,x))| \leq \frac{c}{j}.
    \end{equation*}   

    \noindent For every $x \in \mathring{\mcl Q^j(\R_+^2)}$, $\Lambda^j x \in \mcl Q_\upa(\R^2_+) \cap L^\infty$. Using Proposition~\ref{p.fe is a strong solution}, we deduce that $f^{\perp,j}$ is differentiable on $(0,+\infty) \times \mathring{\mcl Q^j(\R_+^2)}$ and for every $(t,x) \in (0,+\infty) \times \mathring{\mcl Q^j(\R_+^2)}$ we have $\nabla f^{\perp,j}(t,x) = \Lambda_j(\nabla f^\perp(t,\Lambda^jx))$. In addition, we have 
    \begin{align*}
        \partial_t f^{\perp,j}(t,x) &= \partial_t f^\perp(t, \Lambda^jx) \\
                                     &= \int \xi_\perp(\nabla f^\perp(t, \Lambda^jx)) \\
                                     &= H_\perp^j(\nabla f^{\perp,j}(t, x)) + \int\left( \xi_\perp(\nabla f^\perp(t, \Lambda^jx)) - \xi_\perp(\Lambda^j \Lambda_j \nabla f^\perp(t, \Lambda^jx)) \right) .
    \end{align*}
     Since $f^\perp$ is $1$-Lipschitz with respect to $L^1$-norm, we have $|\nabla f^\perp(t,\Lambda^jx))|_{L^\infty} \leq 1$. Now, using Step 1 and Proposition~\ref{p. path approximation} we discover that 
    \begin{align*}
        \left|  \partial_t f^{\perp,j}(t,x) - H_\perp^j(\nabla f^{\perp,j}(t, x)) \right| &= |\xi_\perp(\nabla f^\perp(t, \Lambda^jx)) - \xi_\perp(\Lambda^j \Lambda_j \nabla f^\perp(t, \Lambda^jx)) |_1 \\
                                                                                        &\leq \ell_\perp |\nabla f^\perp(t, \Lambda^jx) - \Lambda^j \Lambda_j \nabla f^\perp(t, \Lambda^jx) |_{L^1} \\
                                                                                        &\leq \frac{c}{j},
    \end{align*}
    where $\ell_\perp = \sup_{|a| \leq 1} |\xi_\perp(a)|$ and $c \geq 0$ is some constant depending on $\xi$ and $|\cdot|$. 
\end{proof}

\begin{theorem}
    The function $f^\perp$ is the unique viscosity solution of
    \begin{equation} \label{e. symmetric HJ}
        \begin{cases} 
            \partial_t u^\perp - \int \xi_\perp(\nabla u^\perp) = 0 \text{ on } (0,+\infty) \times (\mcl Q(\R^2_+) \cap L^2) \\
            u^\perp(0,\cdot) = \psi^\perp.
        \end{cases}
    \end{equation}
   In addition, for every $(t,q) \in (0,+\infty) \times (\mcl Q(\R^2_+) \cap L^2)$, $f^\perp$ admits the Hopf-Lax representation at $(t,q)$,
   \begin{equation} \label{e.symmetric Hopf-Lax}
       f^\perp(t,q) = \sup_{p \in \mcl Q(\R_+^2) \cap L^\infty } \inf_{r \in \mcl Q(\R_+^2) \cap L^\infty} \left\{\psi^\perp(p+q) -\langle p,r \rangle_{L^2}+ t\int_0^1 \xi_\perp \left( r \right) \right\}.
   \end{equation}
\end{theorem}

\begin{remark}
    Taking $q = 0$ in \eqref{e.symmetric Hopf-Lax}, we obtain Theorem~\ref{t.symmetric optimizer matrix}.
\end{remark}

\begin{proof}
    According to \cite[Theorem~4.6]{chen2022hamilton}, \eqref{e. symmetric HJ} has a unique Lipschitz viscosity solution $u^\perp$ and it is given by the variational formula \eqref{e.symmetric Hopf-Lax}. Let us show that $f^\perp$ and $u^\perp$ coincide on $[0,+\infty) \times (\mcl Q(\R^2_+) \cap L^2)$.
    
    Let $j \geq 1$, define Let $\psi^{\perp,j} = f^{\perp,j}(0,\cdot)$. According to Proposition~\ref{p.the approximation are approximate strong solutions}, $\psi^{\perp,j}$ is $(\mcl Q^j(\R^2_+))^*$-nondecreasing and Lipschitz. In addition, $H_\perp^j  \big|_{\mcl Q^j(\R^2_+)}$ is $(\mcl Q^j(\R^2_+))^*$-nondecreasing and locally Lipschitz. According to Theorem~\ref{t. hj cone well posed}, the following Hamilton-Jacobi equation is well posed
    \begin{equation*}
        \begin{cases}
            \partial_t u^{\perp,j} - H_\perp^j(\nabla u^{\perp,j}) = 0 \text{ on } (0,+\infty) \times \mathring{ \mcl Q^j(\R^2_+)} \\
            u^{\perp,j}(0,\cdot) = \psi^{\perp,j}.
        \end{cases}
    \end{equation*}
    We let $u^{\perp,j} \in \mcl V(\mcl Q^j(\R^2_+))$ be its unique viscosity solution.
    
    \noindent \emph{Step 1.} We show that, for every $(t,q) \in \R_+ \times( \mcl Q(\R^2_+) \cap L^\infty)$,
    \begin{equation*}
        f^\perp(t,q) = \lim_{j \to +\infty} u^{\perp,j} (t,\Lambda_j q).
    \end{equation*}
    \noindent According to Proposition~\ref{p. approximate strong solutions are approximate visco solutions} and Proposition~\ref{p.the approximation are approximate strong solutions}, there exists $c > 0$ such that for every $j \geq 1$ and $(t,x) \in \R_+ \times \mcl Q^j(\R^2_+)$, we have 
    \begin{equation}
        |f^{\perp,j}(t,x) -u^{\perp,j}(t,x)| \leq \frac{ct}{j}.
    \end{equation}
    Let $q \in \mcl Q(\R^2_+) \cap L^\infty$, according to Proposition~\ref{p. path approximation} as $j \to +\infty$ we have $\Lambda^j \Lambda_j q \to q$ in $L^1$. By Lipschitz continuity of $f^\perp$ we have
    \begin{equation*}
        f^\perp(t,q) = \lim_{j \to +\infty}  f^{\perp,j}(t,\Lambda_j q) = \lim_{j \to +\infty} u^{\perp,j} (t,\Lambda_j q).
    \end{equation*}
    \noindent \emph{Step 2.} We show that, for every $(t,q) \in \R_+ \times( \mcl Q(\R^2_+) \cap L^\infty)$, 
    \begin{equation*}
        \lim_{j \to +\infty} u^{\perp,j} (t,\Lambda_j q) = u^\perp(t,q).
    \end{equation*}
    \noindent According to \cite[Theorem~1.2~(2)~(d)]{chen2023viscosity}, $u^{\perp,j}$ admits the Hopf-Lax representation. That is, for every $(t,x) \in \R_+ \times \mcl Q^j(\R^2_+)$,
    \begin{equation*}
        u^{\perp,j}(t,x) = \sup_{y \in \mcl Q^j(\R^2_+)} \inf_{z \in \mcl Q^j(\R^2_+)} \left\{\psi^{\perp,j}(x+y) - \langle y , z \rangle_j + t \int \xi_\perp (\Lambda^j z) \right\}.
    \end{equation*}
    Similarly, from \cite[Theorem~1.1]{chen2022hamilton}, we know that for every $(t,q) \in \R_+ \times (\mcl Q(\R^2_+) \cap L^2)$ we have
    \begin{equation*}
        u^{\perp}(t,q) = \sup_{p \in \mcl Q(\R^2_+) \cap L^\infty} \inf_{r \in \mcl Q(\R^2_+) \cap L^\infty} \left\{\psi^{\perp}(q+p) - \langle p , r \rangle_{L^2} + t \int \xi_\perp (r) \right\}.
    \end{equation*}
    \noindent \emph{Step 2.1}  We show that, $\lim_{j \to +\infty} u^{\perp,j} (t,\Lambda_j q) \leq u^\perp(t,q)$.

    \noindent Observe that $\mcl Q^j(\R^2_+) = \{ \Lambda_j p, p \in \mcl Q(\R^2_+) \cap L^\infty \}$. So,
    \begin{equation*}
    \begin{split}
        u^{\perp,j}(t,\Lambda_j q) = \sup_{p \in \mcl Q(\R^2_+) \cap L^\infty} \inf_{r \in \mcl Q(\R^2_+) \cap L^\infty} \biggl\{\psi^{\perp}(\Lambda^j \Lambda_j q + \Lambda^j \Lambda_j p) \\ 
        - \langle\Lambda_j p, \Lambda_j r \rangle_j + t \int \xi_\perp (\Lambda^j \Lambda_j r) \biggl\}.            
    \end{split}
    \end{equation*}
    Since, $\langle\Lambda_j p, \Lambda_j r \rangle_j = \langle \Lambda^j\Lambda_j p,  r \rangle_{L^2}$ and $\{ \Lambda^j\Lambda_j p, p \in \mcl Q(\R^2_+) \cap L^\infty\} \subset \mcl Q(\R^2_+) \cap L^\infty$. We have
    \begin{equation*}
    \begin{split}
        u^{\perp,j}(t,\Lambda_j q) \leq \sup_{p \in \mcl Q(\R^2_+) \cap L^\infty} \inf_{r \in \mcl Q(\R^2_+) \cap L^\infty} \biggl\{\psi^{\perp}(\Lambda^j \Lambda_j q + p) \\- \langle p, r \rangle_{L^2} + t \int \xi_\perp (\Lambda^j \Lambda_j r) \biggl\}.
     \end{split}
    \end{equation*}
    Finally, according to Proposition~\ref{p. apprxoimation decrease H}, we have $\int \xi_\perp (\Lambda^j \Lambda_j r) \leq \int \xi_\perp(r)$, so 
    \begin{equation*}
        u^{\perp,j}(t,\Lambda_j q) \leq \sup_{p \in \mcl Q(\R^2_+) \cap L^\infty} \inf_{r \in \mcl Q(\R^2_+) \cap L^\infty} \left\{\psi^{\perp}(\Lambda^j \Lambda_j q + p) - \langle p, r \rangle_{L^2} + t \int \xi_\perp (r) \right\}.
    \end{equation*}
   Using the Lipschitz continuity of $\psi^\perp$, we discover that
    \begin{equation*}
    \begin{split}
          u^{\perp,j}(t,\Lambda_j q) \leq &|\Lambda^j \Lambda_j q - q|_1 \\ &+ \sup_{p \in \mcl Q(\R^2_+) \cap L^\infty} \inf_{r \in \mcl Q(\R^2_+) \cap L^\infty} \left\{\psi^{\perp}(q + p) - \langle p, r \rangle_{L^2} + t \int \xi_\perp (r) \right\}.       
    \end{split}
    \end{equation*}
    Using Proposition~\ref{p. path approximation}, we obtain $\lim_{j \to +\infty} u^{\perp,j} (t,\Lambda_j q) \leq u^\perp(t,q)$.

    \noindent \emph{Step 2.2}  We show that, $\lim_{j \to +\infty} u^{\perp,j} (t,\Lambda_j q) \geq u^\perp(t,q)$.

    \noindent For every $\varepsilon > 0$, there exists $p \in \mcl Q(\R^2_+) \cap L^\infty$, such that 
    \begin{align*}
        u^\perp(t,q) \leq \varepsilon + \psi^\perp(p+q) + \inf_{r \in \mcl Q(\R^2_+) \cap L^\infty} \left\{ - \langle p,r \rangle_{L^2} +t \int \xi_\perp(r) \right\}.
    \end{align*}
    We have $\{ \Lambda^j z, z \in \mcl Q^j(\R^2_+) \} \subset \mcl Q(\R^2_+) \cap L^\infty$. So,
    \begin{align*}
        u^\perp(t,q) \leq \varepsilon + \psi^\perp(p+q) + \inf_{z \in \mcl Q^j(\R^2_+)} \left\{ - \langle p,\Lambda^j z \rangle_{L^2} +t \int \xi_\perp(\Lambda^j z) \right\}.
    \end{align*}
    Using the Lipschitz continuity of $\psi^\perp$, we have $\psi^\perp(p+q) \leq \psi^{\perp,j}( \Lambda_jp +\Lambda_j q) + |\Lambda^j\Lambda_j (p+q) - (p+q)|_{L^1}$. And we have $\langle p,\Lambda^j z \rangle_{L^2} = \langle \Lambda_j p, z \rangle_j$. So,
    \begin{align*}
        u^\perp(t,q) \leq \varepsilon &+ |\Lambda^j\Lambda_j (p+q) - (p+q)|_{L^1}+ \psi^{\perp,j}( \Lambda_jp +\Lambda_j q) \\
        &+ \inf_{z \in \mcl Q^j(\R^2_+)} \left\{ - \langle \Lambda_j p, z \rangle_j +t \int \xi_\perp(\Lambda^j z) \right\}.
    \end{align*}
    Since, $\{\Lambda_j p , p \in \mcl Q(\R^2_+) \cap L^\infty \} \subset \mcl Q^j(\R^2_+)$, obtain 
     \begin{align*}
        u^\perp(t,q) \leq \varepsilon + |\Lambda^j\Lambda_j (p+q) - (p+q)|_{L^1}+ u^{\perp,j}(t,\Lambda_j q).
    \end{align*}
    Using Proposition~\ref{p. path approximation} and letting $j \to +\infty$, we obtain
    \begin{equation*}
        u^\perp(t,q) \leq \varepsilon + \lim_{j \to +\infty} u^{\perp,j}(t,\Lambda_j q),
    \end{equation*}
    since $\varepsilon > 0$ is arbitrary this concludes Step 2.

    \noindent \emph{Step 3.} Conclusion.

    From Step 1 and Step 2, we deduce that $f^\perp$ and $u^\perp$ coincide on $\R_+ \times( \mcl Q(\R^2_+) \cap L^\infty)$. Since $\mcl Q(\R^2_+) \cap L^\infty$ is dense in $\mcl Q(\R^2_+) \cap L^2$ with respect to $L^1$ convergence and $f^\perp$ and $u^\perp$ are both Lipschitz continuous with respect to $|\cdot|_{L^1}$, we conclude that $f^\perp$ and $u^\perp$ coincide on $\R_+ \times( \mcl Q(\R^2_+) \cap L^2)$.
    \end{proof}

\subsection{The equation for models with vector valued paths} \label{ss.vector paths}

In this section, we repeat the analysis of Section~\ref{ss.matrix paths} under the additional assumption that the interaction function $\xi$ only depends on the diagonal coefficients of the overlap matrix. In this case, we can use the fact the value of the limit free energy is encoded by a Hamilton-Jacobi equation on $\mcl Q(\R^D_+)$, as pointed out in \eqref{e. HJ smooth diag}. In this setting, permutation-invariant paths $q \in \mcl Q(\R^D_+)$ are of the form $q = (p,\dots,p)$ with $p \in \mcl Q(\R_+)$. Leveraging symmetries of the problem, this will allow us to show that the limit free energy is in fact encoded by a Hamilton-Jacobi equation on $\mcl Q(\R_+)$. Since the proofs are very similar, we will not write them out in as many details as in the previous section.

For $p \in \mcl Q(\R_+)$, we let $\mfk q = \text{diag}(p,\dots,p) \in \mcl Q(S^D_+)$ denote the path such that $\mfk q(u)$ is the diagonal matrix with diagonal coefficients $(p(u),\dots,p(u))$. We define $f^\dagger(t,p) = f(t,\text{diag}(p,\dots,p))$, $\psi^\dagger(p) = \psi(\text{diag}(p,\dots,p))$. Recall that here, we identify the function $A \mapsto \xi(A)$ defined on $\R^{D \times D}$ and the function $x \mapsto \xi(\text{diag}(x))$ defined on $\R^D$. With this in mind, we set $\xi_\dagger(\lambda)=  \xi(\lambda,\dots,\lambda)$. 

\begin{proposition} \label{p. strong solution vector}
    The function $f^\dagger$ is Gateaux differentiable at every $(t,p) \in (0,+\infty) \times \left( \mcl Q_\upa(\R_+) \cap L^\infty \right)$ and satisfies
    \begin{equation*}
        \begin{cases}
            \partial_t f^\dagger - \int \xi_\dagger\left(\frac{\nabla f^\dagger}{D}\right) = 0 \text{ on } (0,+\infty) \times \left( \mcl Q_{\upa}(\R_+) \cap L^\infty \right) \\
            f^\dagger(0,\cdot) = \psi^\dagger \text{ on } \mcl Q_{\upa}(\R_+) \cap L^\infty.
        \end{cases}
    \end{equation*}
\end{proposition}

\begin{proof}
    We reproduce the proof of Proposition~\ref{p.fe is a strong solution}, replacing $f$ by $f^\text{diag}$, the function defined by $f^\text{diag}(t,q) = f(t,\text{diag}(q))$ for $q \in \mcl Q(\R^D_+)$. Let $(t,p) \in \R_+ \times \left( \mcl Q_\upa(\R_+) \cap L^\infty \right)$ according to Proposition~\ref{p.permutation invariant fe}, $f$ is Gateaux differentiable at $(t,\text{diag}(p,\dots,p))$ and its gradient is a permutation-invariant path. Therefore, $f^\text{diag}$ is Gateaux differentiable at $(t,(p,\dots,p))$ and its gradient $\nabla f^\text{diag}(t,(p,\dots,p)) \in \mcl Q(\R^D_+)$ satisfies $\nabla f^\text{diag}(t,(p,\dots,p)) = (r,\dots,r)$ for some $r \in \mcl Q(\R_+)$. It follows that $f^\dagger$ is Gateaux differentiable at $(t,p)$ and its Gateaux derivative satisfies $\nabla f^\dagger(t,p) = Dr$. According to \eqref{e. HJ smooth diag}, we have 
    \begin{align*}
        \partial_t f^\dagger(t,p) = \partial_t f^\text{diag}(t,p) = \int \xi(\nabla f^\text{diag}(t,(p,\dots,p))) = \int \xi_\dagger\left(\frac{\nabla f^\dagger(t,p)}{D}\right).
    \end{align*}   
\end{proof}

Recall from \eqref{e.approx cone} that
\begin{equation*}
    \mcl Q^j(\R_+) = \{x \in \R_+^j \big| \; x_1 \leq \dots \leq x_j \}.
\end{equation*}
Also recall that given $x \in \mcl Q^j(\R_+)$, $\Lambda^j x$ denote the path which linearly interpolates between the values $(0,0),(\frac{1}{j},x_1),\dots,(1,x_j)$ and that given $p \in \mcl Q(\R_+)$ the vector $\Lambda_j p \in \mcl Q^j(\R_+)$ is defined so that $\Lambda_j$ and $\Lambda^j$ form an adjoint pair. As previously, we define $f^{\dagger,j}(t,x)=  f^\dagger(t,\Lambda^jx)$ and $H_\dagger^j(x) = \int_0^1 \xi_\dagger(\Lambda^j x)$.

\begin{proposition} \label{p. approximate strong solution vector}
     There exists a constant $c > 0$ such that the following holds. For every $j \geq 1$, the function $f^{\dagger,j}$ is differentiable on $(0,+\infty) \times \mathring{\mcl Q^j(\R_+)}$ and for every $(t,x) \in (0,+\infty) \times \mathring{\mcl Q^j(\R_+)}$, 
   \begin{equation*}
       \left|\partial_t f^{\dagger,j}(t,x) - H_\dagger^j\left(\frac{\nabla f^{\dagger,j}(t,x)}{D}\right)\right| \leq \frac{c}{j}.
   \end{equation*}
   Furthermore, $f^{\dagger,j}  \in \mcl V(\mcl Q^j(\R_+))$.
\end{proposition}

\begin{proof}

    The proof will follow from Proposition~\ref{p.the approximation are approximate strong solutions} and the fact that $f^\dagger(t,p) = f^\perp(t,(p,p))$.
   
    \noindent \emph{Step 1.} We show that $f^{\dagger,j} \in \mcl V(\mcl Q^j(\R_+))$.

    According to Proposition~\ref{p.the approximation are approximate strong solutions}, $f^{\perp,j} \in \mcl V(\mcl Q^j(\R^2_+))$.  In particular,  
    \begin{align*}
      \sup_{\substack{t > 0 \\ x \in \mcl Q^j(\R_+)}} \frac{|f^{\dagger,j}(t,x)-f^{\dagger,j}(0,x)|}{t} &= \sup_{\substack{t > 0 \\ x \in \mcl Q^j(\R_+)}} \frac{|f^{\perp,j}(t,(x,x))-f^{\perp,j}(0,(x,x))|}{t}\\
                                                                                          &\leq \sup_{\substack{t > 0 \\ y \in \mcl Q^j(\R^2_+)}} \frac{|f^{\perp,j}(t,y)-f^{\perp,j}(0,y)|}{t}  \\
                                                                                          &< +\infty.
    \end{align*}
    Similarly,
     \begin{align*}
        \sup_{t > 0} \| f^{\dagger,j}(t,\cdot)\|_{\text{Lip}} &= \sup_{t > 0} \|x \mapsto  f^{\perp,j}(t,(x,x))\|_{\text{Lip}} \\
                                                                      &\leq \sup_{t > 0} \| f^{\perp,j}(t,\cdot)\|_{\text{Lip}} \\
                                                                      &<+\infty.
    \end{align*}
    Furthermore, given $x,x' \in \mcl Q^j(\R_+)$, it is clear that if $x' - x \in (\mcl Q^j(\R_+))^*$, then $(x',x') -(x,x) \in (\mcl Q^j(\R^2_+))^*$. So, 
     \begin{equation*}
        f^{\dagger,j}(t,x') - f^{\dagger,j}(t,x) = f^{\perp,j}(t,(x',x')) - f^{\perp,j}(t,(x,x)) \geq 0.
     \end{equation*}
     Thus, $f^{\dagger,j}(t,\cdot)$ is $(\mcl Q^j(\R_+))^*$-nondecreasing. We have proven that $f^{\dagger,j} \in \mcl V(\mcl Q^j(\R^2_+))$, this concludes Step 1.   
        
    \noindent \emph{Step 2.} We show that there exists $c > 0$ such that for every $j \geq 1$ and  every $(t,x) \in (0,+\infty) \times \mathring{\mcl Q^j(\R_+)}$, $f^{\dagger,j}$ is differentiable at $(t,x)$ and  
    \begin{equation*}
        \left|\partial_t f^{\dagger,j}(t,x) - H_\dagger^j\left(\frac{\nabla f^{\dagger,j}(t,x)}{D}\right)\right| \leq \frac{c}{j}.
    \end{equation*}   

    \noindent For every $x \in \mathring{\mcl Q^j(\R_+)}$, $(x,x) \in \mathring{\mcl Q^j(\R^2_+)}$. Using Proposition~\ref{p.the approximation are approximate strong solutions}, we deduce that $f^{\dagger,j}$ is differentiable on $(0,+\infty) \times \mathring{\mcl Q^j(\R_+)}$ and for every $(t,x) \in (0,+\infty) \times \mathring{\mcl Q^j(\R_+)}$, we have 
    \begin{equation*}
        \partial_t f^{\dagger,j}(t,x) - H^j_\dagger\left(\frac{\nabla f^{\dagger,j}(t,x)}{D}\right) =  \partial_t f^{\perp,j}(t,(x,x)) - H^j_\perp(\nabla f^{\perp,j}(t,(x,x))).
    \end{equation*}
    The result then follows from Proposition~\ref{p.the approximation are approximate strong solutions}.
\end{proof}

\begin{theorem} \label{t.symmetric HJ vector}
    The function $f^\dagger$ is the unique viscosity solution of,
    \begin{equation} \label{e. symmetric HJ vector}
        \begin{cases} 
            \partial_t u^\dagger - \int \xi_\dagger\left(\frac{\nabla u^\dagger}{D}\right) = 0 \text{ on } (0,+\infty) \times \left(\mcl Q(\R_+) \cap L^2 \right)\\
            u^\dagger(0,\cdot) = \psi^\dagger.
        \end{cases}
    \end{equation}
   In addition, for every $(t,r) \in [0,+\infty) \times (\mcl Q(\R_+) \cap L^2)$, $f^\dagger$ admits the Hopf-Lax representation,
   \begin{equation} \label{e.symmetric Hopf-Lax vector}
       f^\dagger(t,r) = \sup_{p \in \mcl Q(\R_+) \cap L^\infty} \left\{\psi^\dagger(p) -t \int_0^1 \xi^* \left( \frac{p-r}{t}, \dots , \frac{p-r}{t} \right) \right\}.
   \end{equation}
\end{theorem}

\begin{remark}
    Taking $r = 0$ in \eqref{e.symmetric Hopf-Lax vector}, we obtain the variational formula \eqref{e.symmetric optimizer vector} of Theorem~\ref{t.symmetric optimizer vector}.
\end{remark}

\begin{remark}
    Here, since \eqref{e. symmetric HJ vector} is written over the set of $1$-dimensional paths, a different version of the Hopf-Lax representation is available \cite[Proposition~A.3]{chen2022hamilton}. This allows us to write the viscosity solution as a “sup” rather than a “sup inf”.
\end{remark}

\begin{remark}
    Reproducing the proof of Proposition~\ref{p.convex dual xi_perp}, we have for every $\lambda \geq 0$,
    \begin{equation*}
        \xi^*(\lambda,\dots,\lambda)= \left( \xi_\dagger \left( \frac{\cdot}{D}\right)\right)^*(\lambda).
    \end{equation*}
\end{remark}


\begin{proof}
    According to \cite[Theorem~4.6]{chen2022hamilton} and \cite[Proposition~A.3]{chen2022hamilton}, \eqref{e. symmetric HJ vector} has a unique Lipschitz viscosity solution $u^\dagger$ and it is given by the variational formula \eqref{e.symmetric Hopf-Lax vector}. Let us show that $f^\dagger$ and $u^\dagger$ coincide on $[0,+\infty) \times (\mcl Q(\R_+) \cap L^2)$.
    
    Let $j \geq 1$, define $\psi^{\dagger,j} = f^{\dagger,j}(0,\cdot)$, according to Proposition~\ref{p. approximate strong solution vector}, $\psi^{\dagger,j}$ is $(\mcl Q^j(\R_+))^*$-nondecreasing and Lipschitz. In addition, $H_\dagger^j  \big|_{\mcl Q^j(\R_+)}$ is $(\mcl Q^j(\R_+))^*$-nondecreasing and locally Lipschitz. Let $u^{\dagger,j} \in \mcl V(\mcl Q^j(\R_+))$ be the unique viscosity solution of
    \begin{equation*}
        \begin{cases}
            \partial_t u^{\dagger,j} - H_\dagger^j\left(\frac{\nabla u^{\dagger,j}}{D}\right) = 0 \text{ on } (0,+\infty) \times \mathring{\mcl Q^j(\R_+)} \\
            u^{\dagger,j}(0,\cdot) = \psi^{\dagger,j}.
        \end{cases}
    \end{equation*}

    \noindent \emph{Step 1.} We show that, for every $(t,p) \in \R_+ \times( \mcl Q(\R_+) \cap L^\infty)$,
    \begin{equation*}
        f^\dagger(t,p) = \lim_{j \to +\infty} u^{\dagger,j} (t,\Lambda_j p).
    \end{equation*}

    \noindent According to Proposition~\ref{p. approximate strong solutions are approximate visco solutions} and Proposition~\ref{p. approximate strong solution vector}, there exists $c > 0$ such that for every $j \geq 1$ and $(t,x) \in \R_+ \times \mcl Q^j(\R_+)$, we have 
    \begin{equation*}
        |f^{\dagger,j}(t,x) -u^{\dagger,j}(t,x)| \leq \frac{ct}{j}.
    \end{equation*}
    Let $p \in \mcl Q(\R_+) \cap L^\infty$, according to Proposition~\ref{p. path approximation}, as $j \to +\infty$ we have $\Lambda^j \Lambda_j p \to p$ in $L^1$. By Lipschitz continuity of $f^\dagger$ we have
    \begin{equation*}
        f^\dagger(t,p) = \lim_{j \to +\infty} f^{\dagger,j}(t,\Lambda_j p) = \lim_{j \to +\infty} u^{\dagger,j} (t,\Lambda_j p).
    \end{equation*}

    \noindent \emph{Step 2.} We show that, for every $(t,p) \in \R_+ \times( \mcl Q(\R_+) \cap L^\infty)$, 
    \begin{equation*}
        \lim_{j \to +\infty} u^{\dagger,j} (t,\Lambda_j p) = u^\dagger(t,p).
    \end{equation*}
    \noindent According to \cite[Theorem~1.2~(2)~(d)]{chen2023viscosity}, $u^{\dagger,j}$ admits the Hopf-Lax representation. That is, for every $(t,x) \in \R_+ \times \mcl Q^j(\R_+)$,
    \begin{equation*}
        u^{\dagger,j}(t,x) = \sup_{y \in \mcl Q^j(\R_+)} \inf_{z \in \mcl Q^j(\R_+)} \left\{\psi^{\dagger,j}(x+y) - \langle y , z \rangle_j + t \int \xi_\dagger \left(\frac{\Lambda^j z}{D}\right) \right\}.
    \end{equation*}
    Similarly, from \cite[Theorem~4.6]{chen2022hamilton}, we know that for every $(t,p) \in \R_+ \times \mcl (Q(\R_+) \cap L^2)$ we have
    \begin{equation*}
        u^{\dagger}(t,p) = \sup_{r \in \mcl Q(\R_+) \cap L^\infty} \inf_{s \in \mcl Q(\R_+) \cap L^\infty} \left\{\psi^{\dagger}(p+r) - \langle r , s \rangle_{L^2} + t \int \xi_\dagger \left(\frac{s}{D}\right) \right\}.
    \end{equation*}
    \noindent \emph{Step 2.1}  We show that, $\lim_{j \to +\infty} u^{\dagger,j} (t,\Lambda_j q) \leq u^\dagger(t,q)$.

    \noindent Observe that $\mcl Q^j(\R_+) = \{ \Lambda_j p, p \in \mcl Q(\R_+) \cap L^\infty \}$. So,
    \begin{equation*}
    \begin{split}
        u^{\dagger,j}(t,\Lambda_j p) = \sup_{r \in \mcl Q(\R_+) \cap L^\infty} \inf_{s \in \mcl Q(\R_+) \cap L^\infty} \biggl\{\psi^{\dagger}(\Lambda^j \Lambda_j p + \Lambda^j \Lambda_j r) \\ - \langle\Lambda_j r, \Lambda_j s \rangle_j + t \int \xi (\Lambda^j \Lambda_j s) \biggl\}.        
    \end{split}
    \end{equation*}
    Since, $\langle\Lambda_j r, \Lambda_j s \rangle_j = \langle \Lambda^j\Lambda_j r,  s \rangle_{L^2}$ and $\{ \Lambda^j\Lambda_j r, r \in \mcl Q(\R_+) \cap L^\infty\} \subset \mcl Q(\R_+) \cap L^\infty$. We have
    \begin{equation*}
    \begin{split}
        u^{\dagger,j}(t,\Lambda_j p) \leq \sup_{r \in \mcl Q(\R_+) \cap L^\infty} \inf_{s \in \mcl Q(\R_+) \cap L^\infty} \bigg\{\psi^{\dagger}(\Lambda^j \Lambda_j p + r) \\ - \langle r, s \rangle_{L^2} + t \int \xi_\dagger \left(\frac{\Lambda^j \Lambda_j s}{D}\right) \bigg\}.       
    \end{split}
    \end{equation*}
    Finally, according to Proposition~\ref{p. apprxoimation decrease H}, we have $\int \xi_\dagger (\Lambda^j \Lambda_j s/D) \leq \int \xi_\dagger(s/D)$, so 
    \begin{equation*}
    \begin{split}
        u^{\dagger,j}(t,\Lambda_j p) \leq \sup_{p \in \mcl Q(\R_+) \cap L^\infty} \inf_{r \in \mcl Q(\R_+) \cap L^\infty} \bigg\{\psi^{\dagger}(\Lambda^j \Lambda_j p + r) \\ - \langle r, s \rangle_{L^2} + t \int \xi_\dagger \left( \frac{s}{D}\right) \bigg\}.       
    \end{split}
    \end{equation*}
   Using the Lipschitz continuity of $\psi^\dagger$, we discover that
    \begin{equation*}
    \begin{split}
          u^{\dagger,j}(t,\Lambda_j p) \leq &|\Lambda^j \Lambda_j p - p|_1 \\ &+ \sup_{p \in \mcl Q(\R_+) \cap L^\infty} \inf_{r \in \mcl Q(\R_+) \cap L^\infty} \biggl\{\psi^{\dagger}(p + r) - \langle r, s \rangle_{L^2} + t \int \xi_\dagger \left( \frac{s}{D}\right) \biggl\}.       
    \end{split}
    \end{equation*}
    Using Proposition~\ref{p. path approximation}, we finally obtain $\lim_{j \to +\infty} u^{\dagger,j} (t,\Lambda_j p) \leq u^\dagger(t,p)$.

    \noindent \emph{Step 2.2}  We show that, $\lim_{j \to +\infty} u^{\dagger,j} (t,\Lambda_j p) \geq u^\dagger(t,p)$.

    For every $\varepsilon > 0$, there exists $r \in \mcl Q(\R_+) \cap L^\infty$, such that 
    \begin{align*}
        u^\dagger(t,p) \leq \varepsilon + \psi^\dagger(p+r) + \inf_{s \in \mcl Q(\R_+) \cap L^\infty} \left\{ - \langle r,s \rangle_{L^2} +t \int \xi_\dagger \left( \frac{s}{D}\right) \right\}.
    \end{align*}
    We have $\{ \Lambda^j z, z \in \mcl Q^j(\R_+) \} \subset \mcl Q(\R_+) \cap L^\infty$. So,
    \begin{align*}
        u^\dagger(t,p) \leq \varepsilon + \psi^\dagger(p+r) + \inf_{z \in \mcl Q^j(\R_+)} \left\{ - \langle r,\Lambda^j z \rangle_{L^2} +t \int \xi_\dagger \left(\frac{\Lambda^j z}{D}\right) \right\}.
    \end{align*}
    Using the Lipschitz continuity of $\psi^\dagger$, we have $\psi^\dagger(p+r) \leq \psi^{\dagger,j}( \Lambda_jp +\Lambda_j r) + |\Lambda^j\Lambda_j (p+r) - (p+r)|_{L^1}$. And we have $\langle r,\Lambda^j z \rangle_{L^2} = \langle \Lambda_j r, z \rangle_j$. So,
    \begin{align*}
        u^\dagger(t,p) \leq \varepsilon &+ |\Lambda^j\Lambda_j (p+r) - (p+r)|_{L^1} + \psi^{\dagger,j}( \Lambda_jp +\Lambda_j r) \\ 
        &+ \inf_{z \in \mcl Q^j(\R_+)} \left\{ - \langle \Lambda_j r, z \rangle_j +t \int \xi_\dagger \left(\frac{\Lambda^j z}{D}\right) \right\}.
    \end{align*}
    Since, $\Lambda_j r \in  \mcl Q^j(\R_+)$, obtain 
     \begin{align*}
        u^\dagger(t,p) \leq \varepsilon + |\Lambda^j\Lambda_j (p+r) - (p+r)|_{L^1}+ u^{\dagger,j}(t,\Lambda_j p).
    \end{align*}
    Using Proposition~\ref{p. path approximation} and letting $j \to +\infty$, we obtain 
    \begin{equation*}
        u^\dagger(t,p) \leq \varepsilon + \lim_{j \to +\infty} u^{\dagger,j}(t,\Lambda_j p),
    \end{equation*}
    since $\varepsilon > 0$ is arbitrary this concludes Step 2.

    \noindent \emph{Step 3.} Conclusion.

    From Step 1 and Step 2, we deduce that $f^\dagger$ and $u^\dagger$ coincide on $\R_+ \times( \mcl Q(\R_+) \cap L^\infty)$. Since $\mcl Q(\R_+) \cap L^\infty$ is dense in $\mcl Q(\R_+) \cap L^2$ with respect to $L^1$ convergence and $f^\dagger$ and $u^\dagger$ are both Lipschitz continuous with respect to $L^1$, we conclude that $f^\dagger$ and $u^\dagger$ coincide on $\R_+ \times( \mcl Q(\R_+) \cap L^2)$.
    \end{proof}

\section{Uniqueness of the optimizer for diagonal models} \label{s.uniqueness}

Let $\mcl P^\upa(\R^D_+)$ denote the set of probability measures on $\R^D_+$ of the form $\text{Law}(q(U))$ with $q \in \mcl Q(\R^D_+)$ and $U$ a uniform random variable on $[0,1)$. The map $q \mapsto \text{Law}(q(U))$ is an isometric bijection from $\mcl Q(\R^D_+)$ to $\mcl P^\upa(\R^D_+)$. Therefore, we can also think of the Parisi functional 
\begin{equation*}
    q \mapsto \psi(q) - t \int_0^1 \xi^*\left( \frac{q}{t}\right),
\end{equation*}
as a functional depending on $\mu \in \mcl P^\upa(\R^D_+)$ rather than $q \in \mcl Q(\R^D_+)$. In a sense, replacing $\mcl Q(\R^D_+)$ by $\mcl P^\upa(\R^D_+)$ changes the geometry. Given $\mu_0,\mu_1 \in \mcl P^\upa(\R^D_+)$ and $q_0,q_1 \in \mcl Q(\R^D_+)$ such that $\mu_i= \text{Law}(q_i(U))$, in general, for $\lambda \in (0,1)$ we have 
\begin{equation*}
    \lambda \mu_1 + (1-\lambda)\mu_0 \neq \text{Law}(\lambda q_1(U)+(1-\lambda)q_0(U)).
\end{equation*}
When $D = 1$, this allows us to reveal a hidden concavity property of the Parisi functional, according to \cite[Theorem~2]{auffinger2015parisi}, the Parisi functional is strictly concave on $\mcl P^\upa(\R_+) = \mcl P(\R_+)$. This is the key property used to establish the uniqueness of Parisi measures for models with $D = 1$ \cite[Corollary~1]{auffinger2015parisi}.

In the setting, $D > 1$ this approach seems to fail. A difficulty is that the convexity of the set $\mcl P^\upa(\R_+)$ is an exception rather than the rule. For every $D > 1$, the set $\mcl P^\upa(\R^D_+)$ is not convex, since for example, $\frac{\delta_{(0,1)}+ \delta_{(1,0)}}{2} \notin \mcl P^\upa(\R^2_+)$. This is problematic, as given two possible maximizing probability measures $\mu_0 \neq \mu_1$ in $\mcl P^\upa(\R^D_+)$, the Parisi functional may not even be defined at $\frac{\mu_0 + \mu_1}{2}$.

In \cite[Theorem~1.2]{chen2023parisipdevector}, it was pointed out that, that the restriction of the Parisi functional to any “one-dimensional subspace of $\mcl P^\upa(\R^D_+)$" is strictly concave. Thanks to Theorem~\ref{t.symmetric optimizer vector}, the Parisi formula can be written as a supremum over probability measures in $\mcl P^\upa(\R^D_+)$ supported on $\{(\lambda,\dots,\lambda),\lambda \geq 0 \}$. This space is of course a “one-dimensional subspace” in the sense of \cite[Theorem~1.2]{chen2023parisipdevector}, this allows us to prove uniqueness of Parisi measures.


Let $\mcl P_1(\R_+)$ denote the set of Borel probability measures on $\R_+$ with finite first moment. Given $\mu \in \mcl P_1(\R_+)$, for every $u \in [0,1)$, we define 
\begin{equation*}
    p_\mu(u) = \inf \{ x \geq 0 \big| \; \mu([0,x]) > u \}.
\end{equation*}
The path $p_\mu \in \mcl Q(\R_+) \cap L^1$ is the quantile function of the probability measure $\mu$ and classically, $\text{Law}(p_\mu(U)) = \mu$.

\begin{proposition}[\cite{chen2023parisipdevector}] \label{p.strict concavity}
    The map $\mu \mapsto \psi^\dagger(p_\mu)$ is strictly concave on $\mcl P_1(\R_+)$.
\end{proposition}

\begin{proof}
    First, note that in \cite{chen2023parisipdevector} the Parisi functional depends on paths in,
    \begin{equation*}
        \Pi(S^D_+) = \{ \pi : [0,1] \to S^D_+ \big| \; \pi \text{ is left-continuous and nondecreasing}\}.
    \end{equation*}
    The Parisi functional is continuous with respect to $|\cdot|_{L^1}$, so one can replace paths $p \in \mcl Q(\R_+) \cap L^\infty$ with their left continuous version without affecting the value of the Parisi functional. Hence, the distinction between $\Pi$ and $\mcl Q$ is simply a matter of taste. Let $\tilde \xi : \R^{D \times D} \to \R$ be a convex function satisfying \eqref{e. covariance} for some Gaussian process $\tilde H_N$. Presumably, as a consequence of \cite[Theorem~1.2]{chen2023parisipdevector}, choosing $\Psi = s \text{id}_D$, we only have that the function $\mu \mapsto \psi(\nabla \tilde \xi \circ \Psi \circ p_\mu)$ is strictly concave on $\mcl P_1(\R_+)$. But, $\psi$ is independent of the choice of $\tilde \xi$, and with $\tilde \xi(R) = \frac{1}{2} \sum_{d,d'=1}^D R^2_{dd'}$ we obtain that $\mu \mapsto \psi^\dagger(p_\mu)$ is strictly concave on $\mcl P_1(\R_+)$.
\end{proof}

\begin{proof}[Proof of Theorem~\ref{t.symmetric optimizer vector}] The variational formula \eqref{e.symmetric optimizer vector} follows from Theorem~\ref{t.symmetric HJ vector}, we now prove existence and uniqueness of an optimizer in \eqref{e.symmetric optimizer vector}. In the first three steps of the proof, we explain that in \eqref{e.symmetric optimizer vector}, the supremum can be taken over a subset of paths bounded in $L^\infty$, this yields existence of an optimizer. The last step is dedicated to showing uniqueness of the optimizer through Proposition~\ref{p.strict concavity}.

    \noindent Step 1. We show that there exists $c \geq 0$, such that for every $\lambda > \lambda' \geq c$ we have 
    \begin{equation*}
        \frac{\xi_\dagger^*(\lambda) - \xi_\dagger^*(\lambda')}{\lambda - \lambda'} \geq 1.
    \end{equation*}
    \noindent Let $a > 0$, there exists $b \geq 0$ such that $\xi_\dagger(\lambda) \leq b$ on $[0,a]$. Let $I_{[0,a]}$ denote the convex function taking the value $0$ on $[0,a]$ and $+\infty$ otherwise. We have $\xi_\dagger \leq b + I_{[0,a]}$ on $\R_+$. Thus, for $\lambda \geq 0$,
    \begin{align*}
        \xi^*_\dagger(\lambda) &= \sup_{\lambda' \geq 0} \left\{ \lambda \lambda ' - \xi_\dagger(\lambda) \right\} \\
                              &\geq  \sup_{\lambda' \geq 0} \left\{ \lambda \lambda ' - b -  I_{[0,a]}\right\} \\
                               &= \sup_{\lambda' \in [0,a]} \left\{ \lambda \lambda ' - b \right\} \\
                               &= a \lambda - b.
    \end{align*}
    This imposes $\liminf_{\lambda \to +\infty} \frac{\xi_\dagger^*(\lambda)}{\lambda} \geq a$. Since $a > 0$ is arbitrary, we have 
    \begin{equation*}
        \liminf_{\lambda \to +\infty} \frac{\xi_\dagger^*(\lambda)}{\lambda} = +\infty.
    \end{equation*}
    Let $c \geq 0$ be so that for every $\lambda \geq c$,
    \begin{equation*}
        \frac{\xi_\dagger^*(\lambda) - \xi_\dagger^*(0)}{\lambda} \geq 1.
    \end{equation*}
    Then, since $\xi_\dagger^*$ is convex, its slope increases and for every $\lambda > \lambda' \geq c$, we have 
    \begin{equation*}
       \frac{\xi_\dagger^*(\lambda) - \xi_\dagger^*(\lambda')}{\lambda-\lambda'} \geq  \frac{\xi_\dagger^*(\lambda) - \xi_\dagger^*(0)}{\lambda} \geq 1.
    \end{equation*}
    This concludes Step 1.

    \noindent Step 2. Let $L^\infty_{\leq tc}$ denote the ball of center $0$ and radius $tc$ in $L^\infty$. We show that
    \begin{equation*}
        \sup_{p \in \mcl Q(\R_+) \cap L^\infty} \left\{ \psi^\dagger(p) - \int \xi^*_\dagger \left( \frac{p}{t}\right) \right\} = \sup_{p \in \mcl Q(\R_+) \cap L^\infty_{\leq tc}} \left\{ \psi^\dagger(p) - \int \xi^*_\dagger \left( \frac{p}{t}\right) \right\}.
    \end{equation*}

    \noindent Let LHS and RHS denote the left-hand side and right-hand side in the previous display. It is clear that LHS $\geq$ RHS, we only prove the other inequality. Let $p \in \mcl Q(\R_+) \cap L^\infty$, and let $\tilde p = p \wedge (t \bar c)$ be the path which coincides with $p$ until $u^* = \inf \{ u \in [0,1), p(u) > tc \}$ and is constant $=tc$ on $[u^*,1)$. We have
    \begin{align*}
        &\left( \psi^\dagger(p) - t \int_0^1 \xi_\dagger^* \left( \frac{p}{t} \right)\right) -  \left( \psi^\dagger(\tilde p) - t \int_0^1 \xi_\dagger^* \left( \frac{\tilde p}{t} \right) \right) \\
        &\leq \int_0^1 |p-\tilde p| - t \int_0^1 \left( \xi^*_\dagger \left( \frac{p}{t} \right) - \xi^*_\dagger \left( \frac{\tilde p}{t} \right) \right) \\
                                                            &= \int_{u^*}^1 p - tc - t \left(  \xi^*_\dagger \left(  \frac{p}{t} \right) - \xi^*_\dagger \left( c \right) \right) \\
                                                            &= \int_{u^*}^1 (p -tc) \left(1 - \frac{ \xi^*_\dagger \left( \frac{p}{t} \right) - \xi^*_\dagger \left( c \right)}{\frac{p}{t} - c} \right) \\
                                                            &\leq 0.
    \end{align*}
    Since $|\tilde p|_{L^\infty} \leq tc$, this concludes Step 2.


    \noindent Step 3. We show that the sup in the variational formula of Theorem~\ref{t.symmetric optimizer vector} is reached at some $p^* \in \mcl Q(\R_+) \cap L^\infty$.

    \noindent The functional $p \mapsto \psi^\dagger(p) - t \int_0^1 \xi^*_\dagger \left( \frac{p}{t}\right)$ is continuous with respect to the $L^1$-norm. So, according to Step 2 it is enough to prove that the set $\mcl Q(\R_+) \cap L^\infty_{\leq tc}$ is compact with respect to $L^1$ convergence. Let $(p_n)_n$ be a sequence of paths in $\mcl Q(\R_+)$ which is bounded by $tc$ in $|\cdot|_{L^\infty}$. By a diagonal argument, up to extraction (not relabelled) we may assume that for every $u \in \Q \cap [0,1)$, we have $\lim_{n \to +\infty} p_n(u) = p(u)$ for some $p(u) \in \R_+$. The map $u \mapsto p(u)$ is nondecreasing and bounded by $tc$ on $\Q \cap [0,1)$. For every $v \in [0,1)$, the quantity $p(u)$ converges as $u \to v$ in $\Q \cap [v,1)$, we denote by $p(v)$ its limiting value. The function $p : [0,1) \to \R_+$ thus defined is càdlàg, nondecreasing and bounded by $tc$. Since $p$ is monotone, the set of discontinuities of $p$ is countable. Let $v \in [0,1)$ be a point at which $p$ is continuous, let $u \in \Q \cap [v,1)$ we have $p_n(v) \leq p_n(u)$, letting $n \to +\infty$ we obtain $\limsup_{n \to +\infty} p_n(v) \leq p(u)$ and since $p$ is continuous at $v$, letting $u \to v$ yields $\limsup_{n \to +\infty} p_n(v) \leq p(v)$. By considering $u \in \Q \cap [0,v]$, we can repeat the same argument to discover that $\liminf_{n \to +\infty} p_n(v) \geq p(v)$. In conclusion, $p_n \to p$ pointwise on $[0,1)$ outside a countable set of points. Since $|p_n|_{L^\infty} \leq tc$, by dominated convergence, it follows that $p_n \to p$ in $L^1$. 

    \noindent Step 4. We show that the sup in the variational formula of Theorem~\ref{t.symmetric optimizer vector} is reached at most at one $p^* \in \mcl Q(\R_+) \cap L^\infty$.
   
    \noindent By contradiction, assume that there are two maximizers $p_0 \neq p_1$ of \eqref{e.symmetric optimizer vector} in $\mcl Q(\R_+) \cap L^\infty$. Let $\mu_i \in \mcl P(\R_+)$ be the law of the random variable $p_i(U)$ where $U$ is a uniform random variable in $[0,1)$. Let $p = p_\mu \in \mcl Q(\R_+) \cap L^\infty$ denote the quantile function of $\mu = \frac{\mu_0+\mu_1}{2} \in \mcl P(\R_+)$. According to Proposition~\ref{p.strict concavity}, we have
    \begin{equation*}
        \psi^\dagger(p) > \frac{\psi^\dagger(p_0)+\psi^\dagger(p_1)}{2}.
    \end{equation*}
    By definition 
    \begin{equation*}
        \int \xi_\dagger^* \left( \frac{p}{t}\right) = \frac{1}{2}\int \xi_\dagger^*\left( \frac{p_0}{t}\right) + \frac{1}{2}\int \xi_\dagger^*\left( \frac{p_1}{t}\right).
    \end{equation*}
    Therefore,
    \begin{equation*}
         \psi^\dagger(p) - t\int \xi^*_\dagger \left( \frac{p}{t}\right) > \frac{1}{2} \left(\psi^\dagger(p_0) - t\int \xi^*_\dagger \left( \frac{p_0}{t}\right) \right) + \frac{1}{2} \left(\psi^\dagger(p_1) - t\int \xi^*_\dagger \left( \frac{p_1}{t}\right) \right).
    \end{equation*}
    This is a contradiction since in the previous display the left-hand side is upper-bounded by $\lim_{N \to +\infty} \bar F_N(t)$ and the right-hand side is equal to  $\lim_{N \to +\infty} \bar F_N(t)$.    
    


\end{proof}

\section{Upper-bound for nonconvex models} \label{s.upper bound}

In this section, we assume that the interaction function $\xi$ only depends on the diagonal coefficients of its argument, and we do \emph{not} assume that $\xi$ is convex. We give a proof of Theorem~\ref{t. upper bound}. We believe that this Theorem is related to the approach developed in \cite{mourrat2020free,mourrat2020nonconvex} where the limit free energy of nonconvex models is lower bounded in terms of the value at $(t,0)$ of the viscosity solution of some Hamilton-Jacobi equation. Theorem~\ref{t. upper bound} will follow from a classic interpolation argument after observing that $\xi$ satisfies the following inequality
\begin{equation*}
    \forall x \in \R^D_+, \; \xi(x) \leq \frac{1}{D} \sum_{d =1}^D \xi(x_d,\dots,x_d).
\end{equation*}

\subsection{An inequality for permutation-invariant covariance functions} \label{ss.inequality}

\begin{proposition} \label{p. inequality monomials}
    Let $i_1,\dots,i_D \in \mathbb{N}$ and $I = \sum_{d = 1}^D i_d$, for every $x_1,\dots,x_D \in \R_+$, we have 
    \begin{equation} \label{e.inequality monomials}
        \frac{1}{D!} \sum_{s \in \msc S_D} \prod_{d = 1}^D x_{s(d)}^{i_d} \leq \frac{1}{D} \sum_{d = 1}^D x^I_d.
    \end{equation}
    In addition, if $I$ is even then \eqref{e.inequality monomials} holds for $x_1,\dots,x_D \in \R$.
 \end{proposition}

\begin{proof}
    If $I = 0$ then the inequality is clear, otherwise by the inequality of arithmetic and geometric means we have 
    \begin{equation*}
        \prod_{d = 1}^D x_{s(d)}^{i_d} \leq \frac{1}{I} \sum_{\delta = 1}^D i_\delta x^I_{s(\delta)}.
    \end{equation*}
    Summing over $s \in \msc S_D$, we obtain 
    \begin{equation*} 
         \frac{1}{D!} \sum_{s \in \msc S_D} \prod_{d = 1}^D x_{s(d)}^{i_d} \leq \frac{1}{I} \sum_{\delta = 1}^D i_\delta \frac{1}{D!} \sum_{s \in \msc S_D} x^I_{s(\delta)}.
    \end{equation*}
    Observing that $\frac{1}{D!} \sum_{s \in \msc S_D} x^I_{s(\delta)} = \frac{1}{D} \sum_{d = 1}^D x^I_d$ is independent of $\delta$, we obtain \eqref{e.inequality monomials}. Assume now that $I$ is even, and let $y \in \R^D$, define $x_d = |y_d|$. Let $E$ be the set of $d \in \{1,\dots,D\}$ such that $y_d < 0$. We have
    \begin{equation*}
         \prod_{d = 1}^D y_{s(d)}^{i_d} = (-1)^{\sum_{s(d) \in E} i_d}  \prod_{d = 1}^D x_{s(d)}^{i_d} \leq  \prod_{d = 1}^D x_{s(d)}^{i_d}.
    \end{equation*}
    Summing over $s \in \msc S_D$ and applying \eqref{e.inequality monomials} to $x$, we obtain 
    \begin{align*}
        \frac{1}{D!} \sum_{s \in \msc S_D} \prod_{d = 1}^D y_{s(d)}^{i_d} &\leq  \frac{1}{D!} \sum_{s \in \msc S_D}  \prod_{d = 1}^D x_{s(d)}^{i_d} \\
                                                                          &\leq  \frac{1}{D} \sum_{d = 1}^D x_{d}^I \\
                                                                          &=  \frac{1}{D} \sum_{d = 1}^D y_{d}^I.
    \end{align*}
    Where the last line follows from the fact that $x_d = \pm y_d$ and $I$ is even.
\end{proof}

Recall that we have defined for every $\lambda \in \R$, $\xi_\dagger(\lambda) = \xi(\lambda,\dots,\lambda)$ and we say that $\xi$ is permutation-invariant when for every $s \in \msc
S_D$ and $x \in \R^D$, $\xi(x^s) = \xi(x)$. Let $A \in \R^{D \times D}$ be a $D \times D$ matrix and $p \in \N^*$, we denote 
\begin{equation*}
    A^{\otimes p} = (A_{i_1,j_1} \times \dots \times A_{i_p,j_p})_{1 \leq i_1,\dots, i_p, j_1,\dots, j_p \leq D} \in \R^{D^p \times D^p}
\end{equation*}
the $p$-fold tensor product of $A$ with itself. For every $C,C' \in \R^{D^p \times D^p}$ we define 
\begin{equation*}
    C \cdot C' = \sum_{1 \leq i_1,\dots, i_p, j_1,\dots, j_p \leq D } C_{(i_1,j_1), \dots, (i_p,j_p)} C'_{(i_1,j_1), \dots, (i_p,j_p)}.
\end{equation*}

\begin{proposition} \label{p.inequality xi}
    Let $\xi : \R^D \to \R$, such that $\xi$ is permutation-invariant and admits an absolutely convergent power series. Assume that for every, $N \geq 1$ there exists a Gaussian process $(H_N(\sigma))_{\sigma \in \R^{D \times N}}$ such that for every $\sigma, \tau \in \R^{D \times N}$,
    \begin{equation*}
        \E \left[ H_N(\sigma) H_N(\tau)\right] = N \xi \left( \frac{\sigma_1 \cdot \tau_1}{N}, \dots, \frac{\sigma_D \cdot \tau_D}{N} \right).
    \end{equation*}
    Then, $\xi_\dagger$ is convex on $\R_+$ and for every $x \in \R^D_+$, we have 
    \begin{equation} \label{e.inequality xi}
        \xi(x) \leq \frac{1}{D} \sum_{d = 1}^D \xi_\dagger(x_d).    
    \end{equation}
\end{proposition}

\begin{remark} \label{r.counterexample}
    For some $\xi$, \eqref{e.inequality xi} is not satisfied on $\R^D$. For example, assume that $D = 2$ and consider $\xi(x_1,x_2) = x_1x_2(x_1+x_2)$, we have $\xi_\dagger(\lambda) = 2\lambda^3$, and
    \begin{equation*}
       \xi(-2,1) =2 > -7 = \frac{\xi_\dagger(-2) + \xi_\dagger(1)}{2}.
    \end{equation*}
\end{remark}

\begin{proof}
    Given $\sigma,\tau \in \R^{D \times N}$, we consider their overlap matrix 
    \begin{equation*}
        \frac{\sigma \tau^*}{N} = \left( \frac{\sigma_d \cdot \tau_{d'}}{N}\right)_{1 \leq d,d' \leq D},
    \end{equation*}
    the diagonal of the overlap matrix $\frac{\sigma \tau^*}{N}$ is the overlap vector
    \begin{equation*}
        R(\sigma,\tau) = \left( \frac{\sigma_1 \cdot \tau_1}{N}, \dots,  \frac{\sigma_D \cdot \tau_D}{N}\right).
    \end{equation*}
    For every $A \in \R^{D \times D}$, we define 
    \begin{equation*}
        \overline{\xi}(A) = \xi(A_{11},\dots,A_{dd}).
    \end{equation*}
    We have $\overline{\xi}(\frac{\sigma \tau^*}{N}) = \xi \left( \frac{\sigma_1 \cdot \tau_1}{N}, \dots, \frac{\sigma_D \cdot \tau_D}{N}\right)$. According to \cite[Poposition~6.6]{mourrat2020free}, there exists a sequence of matrices $(C^{(p)})_{p \geq 1}$ such that $C^{(p)} \in S^{D^p}_+$ and,
    for every $A \in \R^{D \times D}$,
    \begin{equation*}
        \overline \xi(A) = \sum_{p \geq 1} C^{(p)} \cdot A^{\otimes p}.
    \end{equation*}
    Since $\overline \xi(A)$ only depends on the diagonal of $A$, the matrices $C^{(p)}$ must be diagonal with nonnegative coefficients. In particular, there exists nonnegative real numbers $(a_{i_1,\dots,i_D})_{i_1,\dots,i_D \geq 0}$, such that for every $x \in \R^D$,
    \begin{equation*}
        \xi(x) = \sum_{i_1,\dots,i_D \geq 0} a_{i_1,\dots,i_D} \prod_{d = 1}^D x_d^{i_d}.
    \end{equation*}
    Therefore, for every $\lambda \in \R$,
    \begin{equation*}
        \xi_\dagger(\lambda) = \sum_{I \geq 0} \lambda^I \sum_{ \substack{ i_1,\dots,i_D \geq 0 \\ \sum_{\delta= 1}^D i_\delta = I}} a_{i_1,\dots,i_D}
    \end{equation*}
    This shows that $\xi_\dagger$ has a power series expansion with nonnegative coefficients, so $\xi_\dagger$ is convex on $\R_+$. In addition, since $\xi$ is permutation-invariant, given $x \in \R^D_+$ we have 
    \begin{align*}
        \xi(x) &= \frac{1}{D!} \sum_{s \in \msc S_D} \xi(x^s) \\
               &= \sum_{i_1,\dots,i_D \geq 0} a_{i_1,\dots,i_D} \frac{1}{D!} \sum_{s \in \msc S_D} \prod_{d = 1}^D x_{s(d)}^{i_d}.
    \end{align*}
    According to Proposition~\ref{p. inequality monomials}, we thus have 
    \begin{align*}
        \xi(x) &\leq \sum_{i_1,\dots,i_D \geq 0} a_{i_1,\dots,i_D} \frac{1}{D} \sum_{d = 1}^D x_d^{\sum_{\delta = 1}^D i_\delta} \\
               &= \frac{1}{D} \sum_{d =1}^D \sum_{I \geq 0} x_d^I  \sum_{ \substack{ i_1,\dots,i_D \geq 0 \\ \sum_{\delta= 1}^D i_\delta = I}} a_{i_1,\dots,i_D} \\
               &= \frac{1}{D} \sum_{d = 1}^D \xi_\dagger(x_d).
    \end{align*}
\end{proof}

\begin{remark} \label{r.inequality xi everywhere}
    As a consequence of the proof of \eqref{e.inequality xi} we have just given and the second part of Proposition~\ref{e.inequality monomials}, if the power series of $\xi$ only has terms of even degree, then \eqref{e.inequality xi} holds for $x \in \R^D$.
\end{remark}

\subsection{Positivity principle for multiple species}

Roughly speaking, the positivity principle is the statement that if $\sigma,\tau \in \R^{D \times N}$ are independent random variables drawn from the Gibbs measure, then almost surely in the limit $N \to +\infty $ the overlaps $\frac{\sigma_d \cdot \tau_d}{N}$ are all nonnegative. When $D = 1$, it is well known that the positivity principle is verified as soon as the Gibbs measure satisfies the so-called Ghirlanda–Guerra identities \cite[Theorem~3.4]{pan}. When $D > 1$ and $P_N$ is the uniform probability measure on a product of $D$ spheres, it was shown that the statement still holds true \cite[Section~3.2]{bates2022multi}. In this section, for the convenience of the reader and the sake of completeness, we briefly explain how the arguments given to justify \cite[Lemma~3.3]{bates2022multi} can be adapted into a proof of the positivity principle for models with $D > 1$ and $P_N=P_1^{\otimes N}$. The proofs are unchanged, and we simply replace the uniform probability measure on a product of $D$ spheres by $P_1^{\otimes N}$.

Given $\sigma, \tau \in \R^{D \times N}$, we let $R(\sigma,\tau) = (\frac{\sigma_d \cdot \tau_d}{N})_{1 \leq d \leq D}$. Given $w \in [0,1]^D$, we define 
\begin{equation*}
    R^w(\sigma,\tau) = \frac{w \cdot R(\sigma,\tau) }{D} = \frac{1}{D} \sum_{d= 1}^D w_d \frac{\sigma_d \cdot \tau_d}{N}.
\end{equation*}
We let $\msc W =(w_q)_{q \geq 1}$ denote a countable dense subset of $[0,1]^D$. We assume that $\msc W$ contains $e_1\dots,e_D$ and does not contain the null vector. We fix $p \geq 1$, for every $q \geq 1$, we consider the centered Gaussian Process $H_{N,p,q}^\text{pert}$ with covariance
\begin{equation}
    \E H_{N,p,q}^\text{pert}(\sigma) H_{N,p,q}^\text{pert}(\tau)= \frac{N}{4^{p+q}} \left( R^{w_q}(\sigma,\tau) \right)^p.
\end{equation}
Such a process can be explicitly defined by setting
\begin{equation*}
    H_{N,p,q}^\text{pert}(\sigma) = \frac{2^{-(p+q)}}{N^\frac{p-1}{2}D^\frac{p}{2}} \sum_{d_1,\dots,d_p=1}^D \sum_{i_1,\dots,i_p = 1}^N J_{(d_1,\dots,d_p),(i_1,\dots,i_p)} \prod_{k = 1}^p \sqrt{w_{d_k}} \sigma_{d_k,i_k} \tau_{d_k,i_k},
\end{equation*}
Where the coefficients $J_{(d_1,\dots,d_p),(i_1,\dots,i_p)}$ are independent standard normal random variables. We fix $(u_{p,q})_{p,q\geq 1}$ a sequence of numbers in $[1,2]$ and we let
\begin{equation} \label{e.perturbation hamiltonian}
    H_N^\text{pert}(\sigma) = \sum_{p,q \geq 1} u_{p,q} H_{N,p,q}^\text{pert}(\sigma).
\end{equation}
Soon we are going to choose random coefficients $(u_{p,q})_{p,q\geq 1}$ that are independent and uniformly distributed in $[1,2]$ and independent of every other source of randomness. We denote by $\E_u$ the expectation with respect to this law. We let $c_N = N^{-\omega}$ with $0 < \omega < 1/2$, and for every nonrandom function $H : \R^{D \times N} \to \R$, we consider
\begin{equation*}
    \tilde H_N(\sigma) = H(\sigma) +c_N H_N^\text{pert}(\sigma).
\end{equation*}
First, note that by introducing the perturbation $c_N H_N^\text{pert}$ we do not change the limiting value of the free energy. Indeed, by applying Jensen's inequality twice, we obtain
\begin{equation} \label{e. unchanged fe}
   0 \leq \left( - \frac{1}{N} \E \log \int e^{H(\sigma)} \d P_N(\sigma) \right) - \left( - \frac{1}{N} \E \log \int e^{\tilde H(\sigma)} \d P_N(\sigma) \right) \leq \frac{c_N^2}{2}.
\end{equation}
We let $\tilde G_N$ denote the Gibbs measure associated to $\tilde H_N$, that is 
\begin{equation}
    d \tilde G_N(\sigma) \propto e^{\tilde H_N(\sigma)} dP_N(\sigma).
\end{equation}
We denote by $\E_{H_N^\text{pert}}$ the expectation with respect to the randomness of the Gaussian process $H_N^\text{pert}$. We denote by $\tilde G^{\otimes k}_N(A)$ the probability of the event $A$ under the probability measure $\tilde G^{\otimes k}_N$. As usual, let $\sigma,\tau$ denote two random variables sampled independently with law $\tilde G_N$. The following lemma is \cite[Lemma~3.3]{bates2022multi}.

\begin{lemma}[\cite{bates2022multi}] \label{l.positivity principle}
    For every $\varepsilon > 0$,
    \begin{equation*}
        \lim_{N \to +\infty } \sup_{H(\sigma)} \E_u \E_{H_N^\text{pert}} \tilde G_N^{\otimes 2} \left(\exists d, \; \frac{\sigma_d \cdot \tau_d}{N} \leq - \varepsilon \right) = 0,
    \end{equation*}
    where the sup is taken over all nonrandom functions $H : \R^{D \times N} \to \R$ satisfying $\int \exp |H(\sigma)| \d P_N(\sigma) < +\infty$ and $\E_u$ denote the expectation with respect to the parameters $(u_{p,q})_{p,q \geq 1}$ in \eqref{e.perturbation hamiltonian}.
\end{lemma}

\begin{proof}
    It suffices to show that for every $d \in \{1,\dots,D\}$,
    \begin{equation*}
        \lim_{N \to +\infty } \sup_{H(\sigma)} \E_u \E_{H_N^\text{pert}} \tilde G_N^{\otimes 2} \left(\frac{\sigma_d \cdot \tau_d}{N} \leq - \varepsilon \right) = 0.
    \end{equation*}
    Let $(\sigma^\ell)_{\ell \geq 1}$ be independent and identically distributed random variables with law $\tilde G_N$. We let $R_{\ell,\ell'} = R(\sigma^\ell,\sigma^{\ell'})$, and for every $d \in \{1,\dots, D\}$, $R^d_{\ell,\ell'} = \frac{\sigma_d^\ell \cdot \sigma_d^{\ell'}}{N}$. The result will follow from the fact that the array $(R^d_{\ell,\ell'})_{\ell,\ell' \geq 1}$ satisfies the Ghirlanda-Guerra identities in the limit $N \to +\infty$. More precisely, if we define for any $n \geq 1$,  $R(n) = (R_{\ell,\ell'})_{1 \leq \ell,\ell' \leq n}$ and for every $d \in \{1,\dots, D \}$, for every real valued bound measurable function $g  = g(R(n))$ and every continuous function, $h : \R \to \R$,
    \begin{equation*}
    \begin{split}
        \Delta^d(g,n,h) = \biggl|\E \tilde G_N \left( g h(R^d_{1,{n+1}}) \right) - \frac{1}{n} \E \tilde G_N \left( gh(R^d_{1,2}) \right) \\ - \frac{1}{n} \sum_{l = 2}^n \E \tilde G_N \left( gh(R^d_{1,{\ell}}) \right)\biggl|.       
    \end{split}
    \end{equation*}
    Then, according to \cite[Theorem~A.3]{bates2022multi}, we have
    \begin{equation*}
        \lim_{N \to +\infty} \sup_H \E_u \Delta^d(g,n,h) = 0.
    \end{equation*}
   We can then proceed as in \cite[Theorem~3.4]{pan} to deduce the desired result.
\end{proof}

\subsection{The interpolation} \label{ss.interpolation}

Given $\varsigma \in \R^N$, we define 
\begin{equation*}
    H_N^\text{sym}(\varsigma) = H_N(\varsigma,\dots,\varsigma).
\end{equation*}
We let $H_N^{\text{sym},1},\dots,H_N^{\text{sym},D}$ be $D$ independent copies of the Gaussian process $H_N^\text{sym}$ and we assume that $(H_N^{\text{sym},1},\dots,H_N^{\text{sym},D})$ is independent of $H_N$. For every $\sigma = (\sigma_1,\dots,\sigma_D) \in \R^{D \times N}$, we define 
\begin{equation} \label{e. Xi process}
    K_N(\sigma) = \frac{1}{\sqrt{D}} \sum_{d = 1}^D H^{\text{sym},d}_N(\sigma_d).
\end{equation}
The Gaussian process $K_N$ satisfies
\begin{equation*}
    \E \left[ K_N(\sigma) K_N(\tau) \right] = N \Xi \left(R(\sigma,\tau) \right),
\end{equation*}
where $\Xi(x) = \frac{1}{D} \sum_{d = 1}^D \xi_\dagger\left(x_d\right)$. Note that, $\Xi_\dagger$ =$\xi_\dagger$, also note that according to Proposition~\ref{e.inequality xi}, the function $\Xi$ is convex on $\R^D_+$. In particular, the family of Gaussian processes $(K_N)_{N \geq 1}$ is covered by Theorem~\ref{t.symmetric optimizer vector}. Thus, the free energy of the Gaussian process $K_N$ converges when $N \to +\infty$ and the limit free energy can be expressed using a variational formula. We let $\bar F_N$ denote the free energy of $H_N$ as in \eqref{e.free energy}, and we let $\bar G_N$ denote the free energy of $K_N$.

\begin{theorem} \label{t.upperbound}
    Assume that $\xi$ satisfies \eqref{e.inequality xi}, we have for every $t \geq 0$,
    \begin{equation} \label{e.upper bound}
        \limsup_{N \to +\infty} \bar F_N(t) \leq  \lim_{N \to +\infty} \bar G_N(t).
    \end{equation}
\end{theorem}

\begin{remark}
    According to Proposition~\ref{p.inequality xi}, as soon as $\xi$ is permutation-invariant, $\xi$ satisfies \eqref{e.inequality xi}. So we can deduce Theorem~\ref{t. upper bound} from Theorem~\ref{t.upperbound} by applying Theorem~\ref{t.symmetric optimizer vector} to the Hamiltonian $K_N$ and observing that,
    \begin{equation*}
        \left( \xi_\dagger \left( \frac{\cdot}{D}\right)\right)^*(\lambda) = \Xi^*(\lambda,\dots,\lambda).
    \end{equation*}
\end{remark}

\begin{remark}
    Note that in fact, as stated, neither $P_1$ nor $\xi$ are required to be permutation-invariant for Theorem~\ref{t.upperbound} to hold. But if $P_1$ or $\xi$ is not permutation-invariant, there is no reason to believe that \eqref{e.upper bound} is a tight upper bound, and in fact the proof of Theorem~\ref{t.nonconvex fe} given below breaks down in this case.
\end{remark}


\begin{proof}
    Let $H^\text{pert}_N$ be the Gaussian process defined in \eqref{e.perturbation hamiltonian}, we take it independent of $H_N$ and $K_N$. Recall that the covariance of $H^\text{pert}_N$ depends on some parameter $u$ and we denote $\E_u$ the expectation with respect to this parameter. Without loss of generality, throughout the proof we assume that $t =1/2$. For every $\lambda \in [0,1]$, define
    \begin{align*}
        H_{N,\lambda}(\sigma) = &\left( \sqrt{1-\lambda} H_N(\sigma) - (1-\lambda)\frac{N}{2}\xi \left(R(\sigma,\sigma) \right) \right) \\
                                &+ \left( \sqrt{\lambda} K_N(\sigma) - \lambda \frac{N}{2}\Xi \left( R(\sigma,\sigma) \right) \right) \\
                                &+  c_N H^\text{pert}_N(\sigma).
    \end{align*}
    Where $c_N = N^{-\omega}$ with $0 < \omega < 1/2$. We let $\langle \cdot \rangle_\lambda$ denote the Gibbs measure associated to $H_{N,\lambda}(\sigma)$ and we also define the associated free energy
    \begin{equation*}
        \varphi_N(\lambda) = -\frac{1}{N} \E  \log \int \exp \left( H_{N,\lambda}(\sigma) \right) \d P_N(\sigma)
    \end{equation*}
    Here, as previously, we use the symbol $\E$ to denote expectation with respect to every source of randomness involved. In particular, $\E$ integrates out the randomness of $u$ and $\varphi_N(\lambda)$ is a nonrandom quantity. Note that crucially, thanks to \eqref{e. unchanged fe}, the nature of the term $c_N H^\text{pert}_N$ is indeed perturbative as it does not influence the value of the limit free energy of $\langle \cdot \rangle_\lambda$. Inequality \eqref{e.upper bound} can be rewritten as
    \begin{equation*}
        \limsup_{N \to +\infty} \varphi_N(0) \leq  \limsup_{N \to +\infty} \varphi_N(1).
    \end{equation*}
    We will prove this inequality by observing that $\varphi_N$ is almost a nondecreasing function of $\lambda$. For every $\lambda \in (0,1)$, $\varphi_N$ is differentiable at $\lambda$ and we have
    \begin{align*}
        \varphi_N'(\lambda) &= -\frac{1}{N} \E \left\langle \frac{d}{d\lambda} H_{N,\lambda}(\sigma) \right\rangle_\lambda \\
                            &= -\frac{1}{N} \left(  \E \left\langle -\frac{1}{2 \sqrt{1-\lambda}} H_N(\sigma) + \frac{1}{2\sqrt{\lambda}}K_N(\sigma) + \frac{N}{2} (\xi-\Xi) \left( R(\sigma,\sigma)\right) \right\rangle_\lambda \right),
    \end{align*}
    Let us now treat the first term in the above display using Gaussian integration by parts \cite[Lemma~1.1]{pan}. For every $\sigma,\tau \in \R^D$, we have 
    \begin{align*}
        &\E \left[ \left( -\frac{1}{2 \sqrt{1-\lambda}} H_N(\sigma) + \frac{1}{2\sqrt{\lambda}}K_N(\sigma) \right)H_{N,\lambda}(\tau) \right] \\        
        &= \E \left[ -\frac{1}{2 \sqrt{1-\lambda}} H_N(\sigma) \sqrt{1-\lambda} H_N(\tau) \right] + \E \left[ \frac{1}{2\sqrt{\lambda}}  K_N(\sigma) \sqrt{\lambda} K_N(\tau) \right] \\
        &+ \E[c_N H_N^\text{pert}(\sigma)] \E \left[ -\frac{1}{2 \sqrt{1-\lambda}} H_N(\sigma) + \frac{1}{2\sqrt{\lambda}}K_N(\sigma) \right] \\
        &+ \left((1-\lambda)\xi + \lambda \Xi \right)\E\left[-\frac{1}{2 \sqrt{1-\lambda}} H_N(\sigma) + \frac{1}{2\sqrt{\lambda}}K_N(\sigma)\right] \\
                                                                                      &= -\frac{1}{2} N \xi \left(R(\sigma,\tau)\right) + \frac{1}{2} N \Xi \left( R(\sigma,\tau) \right) + 0  + 0\\
                                                                                      &= \frac{N}{2} (\Xi-\xi)  \left( R(\sigma,\tau)\right).
    \end{align*}
    Thus, the Gaussian integration by parts formula yields
    \begin{equation*}
        \varphi'_N(\lambda) = \frac{1}{2}  \E \left\langle (\Xi-\xi)  \left( R(\sigma,\tau) \right) \right\rangle. 
    \end{equation*}
    From Remark~\ref{r.inequality xi everywhere}, if the power series that defines the function $\xi$ does not include any term of odd degree, then $\Xi-\xi$ is nonnegative on $\R^D$ which implies that $\varphi'_N(\lambda) \geq 0$ and we do not need the perturbation $c_N H^\text{pert}_N(\sigma)$ in this case. However, if some terms of odd degree are present, then according to Proposition~\ref{p.inequality xi} and Remark~\ref{r.counterexample}, we can only claim that $\Xi -\xi$ is nonnegative on $\R_+^D$. Since the overlaps $\frac{\sigma_d \cdot \tau_d}{N}$ can take values in $[-1,1]$, the previous display only implies 
    \begin{equation*}
        \varphi'_N(\lambda) \geq \delta(\varepsilon) - \frac{\| \Xi-\xi \|_\infty}{2} \E \left\langle \mathbf{1}_{\exists d, \; \frac{\sigma_d \cdot \tau_d}{N} \leq - \varepsilon} \right\rangle_\lambda,
    \end{equation*}
    where $\delta(\varepsilon) = \min \left\{ \frac{(\Xi-\xi)(R)}{2} \big| \; \forall d, \; R_d \in (-\varepsilon,1] \right\}$. Since $\Xi-\xi$ is locally Lipschitz and nonnegative on $\R_+^D$, there exists $C > 0$ such that $\delta(\varepsilon) \geq -C\varepsilon$. Therefore, 
    \begin{equation*}
        \varphi_N(1) \geq \varphi_N(0) - C\varepsilon - \frac{\| \Xi-\xi \|_\infty}{2} \int_0^1 \E \left\langle \mathbf{1}_{\exists d, \; \frac{\sigma_d \cdot \tau_d}{N} \leq - \varepsilon} \right\rangle_\lambda \d\lambda.
    \end{equation*}
    To control the remaining term, observe that
    \begin{align*}
        \E \left\langle \mathbf{1}_{\exists d, \; \frac{\sigma_d \cdot \tau_d}{N} \leq - \varepsilon} \right \rangle_\lambda \leq \sup_H \E_u\E_{H_N^\text{pert}} \tilde G_N^{\otimes 2} \left( \exists d, \; \frac{\sigma_d \cdot \tau_d}{N} \leq -\varepsilon \right),
    \end{align*}
    where, as in Lemma~\ref{l.positivity principle}, the supremum is taken over all nonrandom Hamiltonians $H$, $\E_{H_N^\text{pert}}$ denotes the expectation with respect to $H_N^\text{pert}$ and $\tilde G_N$ is the Gibbs measure of $H + c_N H_N^\text{pert}$. In particular, the right-hand side in the previous display is independent of $\lambda$ and according to Lemma~\ref{l.positivity principle}, it vanishes in the limit $N \to +\infty$. Letting $N \to +\infty$, we obtain 
    \begin{equation*}
       \limsup_{N \to +\infty} \varphi_N(1) \geq \limsup_{N \to +\infty} \varphi_N(0) - C\varepsilon - 0.
    \end{equation*}
    Since $\varepsilon > 0$ is arbitrary, we can conclude by letting $\varepsilon \to 0$.
    \end{proof}

\subsection{Connections with the Hamilton-Jacobi approach} \label{ss. nonconvex hj}

When $\xi$ is nonconvex, the Parisi formula completely breaks down. To the best of our knowledge, until recently, it seems that there was no clear conjecture on what the limit of the free energy should be in this case. In \cite[Conjecture~2.6]{mourrat2019parisi}, it is proposed that results such as Theorem~\ref{t.visco and fe} should generalize to nonconvex models. It was later shown in \cite{mourrat2020free,mourrat2020nonconvex} that the $\liminf$ of $\bar F_N(t)$ as $N \to +\infty$ is lower bounded in terms of the viscosity solution of a Hamilton-Jacobi equation. This lower bound holds, regardless of convexity and permutation invariance. 

Recall that $\xi_\dagger(\lambda) = \xi(\lambda,\dots,\lambda)$ and $\psi^\dagger(p) = \psi(p,\dots,p)$, also recall that 
\begin{equation*}
    \Xi(x) = \frac{1}{D} \sum_{d = 1}^D \xi(x_d,\dots,x_d).
\end{equation*}
If we combine the lower bound of \cite{mourrat2020free,mourrat2020nonconvex} with Theorem~\ref{t.upperbound} and \ref{t.symmetric HJ vector}, we obtain the following proposition. 
\begin{proposition}
Assume that $\xi$ is permutation-invariant and only depends on the diagonal coefficients of its argument, assume that $P_1$ is permutation-invariant. Then, even when $\xi$ is nonconvex on $S^D_+$, we have 
\begin{equation} \label{e.upper and lower}
    f(t,0) \leq \liminf_{N \to +\infty} \bar F_N(t) \leq \limsup_{N \to +\infty} \bar F_N(t) \leq g(t,0),
\end{equation}
where $f$ and $g$ are the viscosity solutions of the following equations,
\begin{equation} \label{e.HJ nonconvex}
    \begin{cases}
        \partial_t f - \int \xi(\nabla f) = 0 \text{ on } (0,\infty) \times (\mcl Q(\R^D_+) \cap L^2) \\
        f(0,\cdot) = \psi \text{ on }\mcl Q(\R^D_+) \cap L^2,
    \end{cases}
\end{equation}

\begin{equation} \label{e.HJ nonconvex vector}
    \begin{cases}
        \partial_t g - \int \xi_\dagger\left(\frac{\nabla g}{D}\right) = 0 \text{ on } (0,+\infty) \times (\mcl Q(\R_+) \cap L^2) \\
        g(0,\cdot) = \psi^\dagger \text{ on }\mcl Q(\R_+) \cap L^2.
    \end{cases}
\end{equation}
\end{proposition}
The results of Section~\ref{ss.vector paths} can be regarded as a proof of the fact $f(t,0) = g(t,0)$ when $\xi$ is convex. To prove this statement we only relied on the convexity of $\xi_\dagger$ and the fact that $f$ is not simply a solution in the viscosity sense but also a Gateaux differentiable solution. To the best of our knowledge, there is no known proof of the fact that $f$ is Gateaux differentiable when $\xi$ is nonconvex. Indeed, in the convex case, proving that $f$ is Gateaux differentiable amounts to observing that it is semi-convex thanks to the Hopf-Lax representation and semi-concave thanks to the fact it can be written as the limit of the enriched free energy (see \cite[Proposition~8.6]{chenmourrat2023cavity}). Since the Hopf-Lax representation for viscosity solutions is not available when $\xi$ is nonconvex, some innovative new ideas seem to be needed to prove (or disprove) the Gateaux differentiability of $f$ in the nonconvex case. Nevertheless, for nonconvex models, we can formulate the following hypothesis.

\begin{equation}     
 \label{e.hypo} \tag{H$_{\xi,P_1}$}
    \begin{split}
         \text{The viscosity solution of \eqref{e.HJ nonconvex} is Gateaux} \\
         \text{differentiable on $(0,\infty) \times \left( \mcl Q_\upa(\R^D_+) \cap L^\infty \right)$.}
    \end{split}   
\end{equation}

In physics terms, \eqref{e.hypo} can be regarded as the statement that no first order phase transition occurs in the mean-field spin glass model with covariance function $\xi$ and reference measure $P_1$. According to \cite[Propositions~7.2~\&~8.6]{chenmourrat2023cavity}, when $\xi$ is convex, \eqref{e.hypo} holds. When $\xi$ is nonconvex, assuming \eqref{e.hypo} holds, we obtain the following representation for the limit free energy.

\begin{theorem} \label{t.nonconvex fe}
    Assume that $\xi$ is permutation-invariant and only depends on the diagonal coefficients of its argument, assume that $P_1$ is permutation-invariant. If \eqref{e.hypo} is true, then for every $t > 0$, the free energy $\bar F_N(t)$ converges as $N \to +\infty$ and 
    \begin{equation} \label{e.nonconvex fe}
        \lim_{N \to +\infty} \bar F_N(t) = g(t,0),
    \end{equation}
    where $g$ is the viscosity solution of \eqref{e.HJ nonconvex vector}. Furthermore, we have $g(t,0) = f(t,0)$, where $f$ is the viscosity solution of \eqref{e.HJ nonconvex} and 
    \begin{equation*}
        g(t,0) = \sup_{p \in \mcl Q(\R_+) \cap L^\infty} \left\{ \psi(p \text{id}_D) - t \int_0^1 \Xi^* \left( \frac{ p(u) \text{id}_D}{t} \right) \d u \right\}.
    \end{equation*}
\end{theorem}
\begin{remark} \label{r.before t_c}
    In fact, even if \eqref{e.hypo} only holds on a region of the form $[0,t_c) \times (\mcl Q_\upa(\R^D_+) \cap L^\infty)$ for some $t_c > 0$ then \eqref{e.nonconvex fe} still holds for $t \leq t_c$.
\end{remark}
\begin{remark} \label{r.no first order transition for small t}
     It seems plausible that there exists $t_c > 0$ such that the weaker version of \eqref{e.hypo} mentioned in Remark~\ref{r.before t_c} holds. Indeed, $\psi$ is Gateaux differentiable on $\mcl Q(\R^D_+) \cap L^2$ and its gradient is a Lipschitz function \cite[Corollary~5.2]{chenmourrat2023cavity}. Therefore, the characteristic curves
     \begin{equation*}
         q \mapsto q - t\nabla \xi(\nabla \psi(q)),
     \end{equation*}
     are injective when $t \geq 0$ is small enough $(t < 1/ \|\nabla \xi(\nabla \psi) \|_{\text{Lip}})$. Usually, the viscosity solution remains differentiable as long as the characteristic curves are injective. 
\end{remark}
When $D = 1$, the Parisi formula is a differentiable function of $\beta = \sqrt{2t}$ \cite{panchenko2008differentiability}. Therefore, the right-hand side of \eqref{e.nonconvex fe} is differentiable on $(0,+\infty)$ as a function of $t$. So at least $\eqref{e.hypo}$ is not in direct contradiction with \eqref{e.nonconvex fe}. 

Conditionally on \eqref{e.hypo}, Theorem~\ref{t.nonconvex fe} confirms \cite[Conjecture~2.6]{mourrat2019parisi} for permutation-invariant models. Note that the argument given below will not show that $(t,q) \mapsto \lim_{N \to +\infty} \bar F_N(t,q)$ is the viscosity solution of 
\begin{equation*}
    \begin{cases}
        \partial_t f - \int \xi(\nabla f) = 0 \text{ on } (0,\infty) \times (\mcl Q(\R^D_+) \cap L^2) \\
        f(0,\cdot) = \psi \text{ on }\mcl Q(\R^D_+) \cap L^2.
    \end{cases}
\end{equation*}
In fact, even under \eqref{e.hypo}, it does not follow from the proof below that $\bar F_N(t,q)$ converges as $N \to +\infty$ when $q$ is not of the form $(p,\dots,p)$.

It is worth mentioning that Theorem~\ref{t.nonconvex fe} still holds if \eqref{e.hypo} is replaced by the assumption that the enriched free energy converges as $N \to +\infty$ and the limit is Gateaux differentiable on $(0,\infty) \times \left( \mcl Q_\upa(\R^D_+) \cap L^\infty \right)$. To prove this version of Theorem~\ref{t.nonconvex fe}, one can adapt the argument given below, replacing $f$ the viscosity solution of \eqref{e.HJ nonconvex} by the limit free energy and appealing to \cite[Proposition~7.2]{chenmourrat2023cavity}. Note that when choosing to state Theorem~\ref{t.nonconvex fe} this way, the convergence of the free energy as $N \to +\infty$ is a hypothesis, not a conclusion.
   
\begin{proof}[Proof of Theorem~\ref{t.nonconvex fe}]
    Let $f$ and $g$ denote the viscosity solutions of \eqref{e.HJ nonconvex} and \eqref{e.HJ nonconvex vector}. Given \eqref{e.upper and lower}, it suffices to verify that $f(t,0) = g(t,0)$ to prove the desired result. As previously, for every $p \in \mcl Q(\R_+) \cap L^1$, we let 
    \begin{equation*}
        f^\dagger(t,p) = f(t,(p,\dots,p)).
    \end{equation*}

    \noindent Step 1. We show that for every $s \in \msc S_D$, 
    \begin{equation*}
        f(t,(p_1,\dots,p_D)) = f(t,(p_{s(1)},\dots,p_{s(D)})).
    \end{equation*}

    \noindent Given $s \in \msc S_D$, we let 
    \begin{align*}
        f^s(t,(p_1,\dots,p_D)) &= f(t,(p_{s(1)},\dots,p_{s(D)})) \\
        \psi^s(p_1,\dots,p_D) &= \psi(p_{s(1)},\dots,p_{s(D)}).        
    \end{align*}
     Since $\psi$ depends only on $P_1$ and not on $\xi$, according to Proposition~\ref{p. fe is perm invariant}, we have $\psi^s = \psi$. Using the permutation invariance of $\xi$, it can be checked that $f^s$ is a viscosity solution of \eqref{e.HJ nonconvex}. By uniqueness of the viscosity solution of \eqref{e.HJ nonconvex}, we obtain $f^s = f$.

    \noindent Step 2. We show that $f^\dagger$ is a strong solution of
    \begin{equation*}
        \begin{cases}
            \partial_t f^\dagger - \int \Xi_\dagger\left(\frac{\nabla f^\dagger}{D}\right) = 0 \text{ on } (0,+\infty) \times \left( \mcl Q_{\upa}(\R_+) \cap L^\infty \right) \\
            f^\dagger(0,\cdot) = \psi^\dagger \text{ on } \mcl Q_{\upa}(\R_+) \cap L^\infty.
        \end{cases}
    \end{equation*}

    \noindent Using Step 1, we can proceed as in the proof of Proposition~\ref{p. strong solution vector} to discover that for every $(t,p) \in (0,+\infty) \times (\mcl Q_{\upa}(\R_+) \cap L^\infty)$, $f^\dagger(t,\cdot)$ is Gateaux differentiable at $p$ and
    \begin{equation}
        \nabla f(t,(p,\dots,p)) = \left(\frac{\nabla f^\dagger(t,p)}{D},\dots,\frac{\nabla f^\dagger(t,p)}{D}\right).
    \end{equation}
    It then follows from \eqref{e.hypo} that
    \begin{equation*}
        \partial_t f^\dagger(t,p) - \int \xi_\dagger\left(\frac{\nabla f^\dagger(t,p)}{D}\right) = \partial_t f(t,(p,\dots,p)) - \int \xi(\nabla f(t,(p,\dots,p))) = 0.
    \end{equation*}
    Finally, since $\xi_\dagger =  \Xi_\dagger$ we conclude Step 2.

    \noindent Step 3. We show that $f(t,0) = g(t,0)$.

    \noindent According to Proposition~\ref{p.inequality xi}, $\Xi$ is convex on $\R^D_+$ and $\Xi$ is the covariance function of the process defined by \eqref{e. Xi process}. In Theorem~\ref{t.symmetric HJ vector} we have shown that, when the nonlinearity is convex, any strong solution of a Hamilton-Jacobi equation on $\mcl Q_{\upa}(\R_+) \cap L^\infty$ is also a viscosity solution on $\mcl Q(\R_+) \cap L^2$. Therefore, using Step 2 and applying Theorem~\ref{t.symmetric HJ vector} to $\Xi_\dagger$ and $f^\dagger$, we discover that $f^\dagger$ the viscosity solution of 
    \begin{equation*}
    \begin{cases}
        \partial_t f^\dagger - \int \Xi_\dagger\left(\frac{\nabla  f^\dagger}{D}\right) = 0 \text{ on } (0,+\infty) \times (\mcl Q(\R_+) \cap L^2) \\
         f^\dagger(0,\cdot) = \psi^\dagger \text{ on }\mcl Q(\R_+) \cap L^2.
    \end{cases}
    \end{equation*}
    Once again, since $\Xi_\dagger = \xi_\dagger$, $f^\dagger$ is in fact the viscosity solution of \eqref{e.HJ nonconvex vector}. Finally, by uniqueness of the viscosity solution, it follows that $f^\dagger = g$ and choosing $p = 0$, we deduce $f(t,0)= g(t,0)$. Thus $ F_N(t)$ converges as $N \to +\infty$ and 
    \begin{equation*}
        \lim_{N \to +\infty} \bar F_N(t) = f(t,0) = g(t,0).
    \end{equation*}

    \noindent Step 4. We show that 
    \begin{equation*}
        g(t,0) = \sup_{p \in \mcl Q(\R_+) \cap L^\infty} \left\{ \psi(p \text{id}_D) - t \int_0^1 \Xi^* \left( \frac{ p(u) \text{id}_D}{t} \right) \d u \right\}.
    \end{equation*}
    \noindent According to Proposition~\ref{p.inequality xi}, $\xi_\dagger$ is convex on $\R_+$. Therefore, it follows from the one dimensional Hopf-Lax representation \cite[Proposition~A.3]{chen2022hamilton} that,
    \begin{equation*}
        g(t,0) = \sup_{p \in \mcl Q(\R_+) \cap L^\infty} \left\{ \psi^\dagger(p) - t \int_0^1 \left( \xi_\dagger \left( \frac{\cdot}{D}\right)\right)^* \left( \frac{ p(u)}{t} \right) \d u \right\}.
    \end{equation*}
    In addition, $\Xi$ is convex on $\R_+^D$ and we have
    \begin{equation*}
        \left( \xi_\dagger \left( \frac{\cdot}{D}\right)\right)^*(\lambda) = \Xi^*(\lambda,\dots,\lambda).
    \end{equation*}
    Thus, 
    \begin{equation*}
        g(t,0) = \sup_{p \in \mcl Q(\R_+) \cap L^\infty} \left\{ \psi(p \text{id}_D) - t \int_0^1 \Xi^* \left( \frac{ p(u) \text{id}_D}{t} \right) \d u \right\}.
    \end{equation*}
\end{proof}
\newpage
\small
\bibliographystyle{plain}
\bibliography{ref}

\end{document}